\documentclass[11pt]{article}
\usepackage[utf8]{inputenc}
\usepackage{amssymb}
\usepackage{amsmath}
\usepackage{amsthm}
\usepackage{hyperref}
\usepackage{a4wide}
\usepackage{breqn}
\usepackage[table,xcdraw]{xcolor}
\usepackage{xcolor}
\usepackage{mathtools}
\usepackage{bm}
\usepackage{multirow}
\usepackage{enumitem}
\usepackage{mathrsfs}
\usepackage{graphicx}
\usepackage{caption}
\usepackage{url}
\usepackage{placeins}
\usepackage[english]{babel}
\usepackage{array}
\usepackage{subcaption}
\usepackage{booktabs}
\usepackage[numbers ]{natbib}

\usepackage{graphics}
\usepackage{lineno} 
\usepackage{graphicx}
\numberwithin{equation}{section}
\newcommand{\dm}[1]{{\displaystyle{#1}}}

\newtheorem{theorem}{\bf Theorem}[section]
\newtheorem{definition}{Definition}[section]
\newtheorem{corollary}{Corollary}[section]
\newtheorem{proposition}{Proposition}[section]
\newtheorem{lemma}{Lemma}[section]
\newtheorem{remark}{Remark}[section]
\theoremstyle{remark}
\newtheorem{example}{\bf Example}
\newcommand{\vone}{\vskip 2ex}

\def \vec{\mathrm v\mathrm e \mathrm c}

\def \f{\widetilde{f}}

\def \R{{\mathbb R}}
\def \K{{\mathscr{K}}}
\def \C{{\mathscr{C}}}
\def \Ms{{\mathscr{M}}}

\def \T{\mathsf T}
\def\bmatrix#1{\left[\begin{matrix}
		#1
	\end{matrix}\right]}

\def \M{{\mathcal M}}
\def \E{{\mathcal E}}

\def \SY{{\mathcal S }}
\def \vp{{\varphi}}
\def \D{{\Delta}}

\def \d{{\bf d}}

\def \f{{\bf f}}
\def \g{{\bf g}}

\def \vr{{\Bigl\vert }}
\def \Vr{{\Bigl\Vert }}
\def \L{\mathcal{L}}

\def \Up{{\Upsilon}}

\def\bmatrix#1{\left[ \begin{matrix} #1 \end{matrix} \right]}
\def \noin{\noindent}

\def \R{{\mathbb R}}

\title{
Structured condition numbers for a linear function of the solution of the generalized saddle point problem}
\author{Sk. Safique Ahmad\footnotemark[2] \footnotemark[1] \and Pinki Khatun \footnotemark[2] }
\date{}
\begin{document}

	\maketitle	
		
		\begin{abstract}
  This paper addresses structured normwise,  mixed, and componentwise condition numbers (\textit{CNs}) for a linear function of the solution to the generalized saddle point problem (\textit{GSPP}).
  We present a general framework that enables us to measure the structured \textit{CNs} of the individual components of the solution. Then, we derive their explicit formulae when the input matrices have symmetric, Toeplitz, or some general linear structures.  In addition, compact formulae for the unstructured \textit{CNs} are obtained, which recover previous results on \textit{CNs} for \textit{GSPPs} for specific choices of the linear function. Furthermore,  applications of the derived structured \textit{CNs} are provided to determine the structured \textit{CNs} for the weighted Toeplitz regularized least-squares problems and Tikhonov regularization problems, which retrieves some previous studies in the literature.
		\end{abstract}
		\noindent {\bf Keywords.} 
		Generalized	saddle point problems,	Condition number, Perturbation analysis,  Weighted Teoplitz regularized least-squares problem, Teoplitz matrices
  
			\noindent {\bf AMS subject classification.}  65F35, 65F20, 65F99, 15A12

	\footnotetext[2]{
		Department of Mathematics, Indian Institute of Technology Indore, Khandwa Road, Indore, 453552, Madhya Pradesh, India. Email: \texttt{safique@iiti.ac.in} (Sk. Safique Ahmad), \texttt{phd2001141004@iiti.ac.in},  \texttt{pinki996.pk@gmail.com} (Pinki Khatun)}
	\footnotetext[1]{Corresponding author.
 
	Pinki Khatun acknowledges support from Council of Scientific $\&$ Industrial Research (CSIR), New Delhi, in the form of a fellowship (File no. $09/1022(0098)/2020$-EMR-I).}
	
	
\section{Introduction}
Generalized saddle point problems (\textit{GSPPs}) have received significant attention owing to their extensive applications across numerous fields in scientific computing such as computational fluid dynamics \cite{navier2015, Elman2005}, constrained optimization \cite{OPTM2014, OPTM1992}, and so on. Consider the following two-by-two block linear system: 
	\begin{align}\label{eq11}
		&\mathcal{M}z\equiv	\bmatrix{A & B^{\T}\\ C & D}\bmatrix{x \\ y}=\bmatrix{f \\ g}\equiv \bf{ d},
	\end{align}
	where $A\in \R^{m\times m},$ $B,C\in\R^{n\times m},$ $D\in \R^{n\times n},$ $x,f\in \R^m,$ $y,g\in \R^n,$ and  $B^{\T}$ represents the transpose of any matrix $B$. Then (\ref{eq11}) is referred to as a \textit{GSPP} if the block matrices $A, B, C$ and $D$ satisfy some special structures, such as $B=C,$  symmetric, Toeplitz, or have some other linear structures \cite{Benzi2005}. Recently, a large amount of efficient iteration methods have been proposed to solve the linear system
 \eqref{eq11}, such as inexact Uzawa schemes \cite{ZZBaiUzawa}, Krylov subspace methods \cite{Krylov2002}, and so on. For a comprehensive survey of applications, algebraic properties, and iteration methods for \textit{GSPPs}, we refer to \cite{ZZBai2021} and reference therein.

 The \textit{GSPP} or its special cases originate from a wide range of applications. For example: (i) The Karush-Kuhn-Tucker (KKT) system ($A=A^{\T},$ $B=C$, and $D={\bf 0},$ here ${\bf 0}$ denotes the zero matrix of appropriate dimension) is one of the simplest versions of $(\ref{eq11})$ and arises from the KKT first-order optimality condition in constrained optimization problems \cite{ Benzi2005, Sun1999}. (ii) The sinc-Galerkin discretization of ordinary differential equations (ODEs) leads to a problem of the form \eqref{eq11} \cite{ZZBai2013, ZZBai2011, ZZBai2014}. (iii) The system \eqref{eq11} also comes from the finite element discretization of time-harmonic eddy current models \cite{ZZBai2012BASI}. (iv) The finite difference discretization of time-dependent Stokes equations generates systems in the form of \eqref{eq11} \cite{ZZBai2012SOR}. (v) \textit{GSPPs} emerge in the weighted Toeplitz regularized least-squares (\textit{WTRLS}) problem \cite{Benziimage2006} arising from image restoration and reconstruction problems \cite{imagereconstruction, IMAGERESTORATION} of the form:
 \begin{align}\label{eq:WTRLS}
     \min_{y\in \R^n}\|\mathbf{M}y-\tilde{\bf d}\|_2^2,
 \end{align}
	where $\mathbf{M}=\bmatrix{W^{1/2}Q\\ \sqrt{\lambda} I_n}\in\R^{(m+n)\times n}$, ${\bf \widetilde{d}}=\bmatrix{W^{1/2}f\\ {\bf 0}}\in\R^{m\times n},$ $Q\in \R^{m\times n} (m\geq n)$ is a full rank Toeplitz matrix and $W\in \R^{m\times m}$ is a symmetric positive definite weighted matrix. The equivalent augmented system possesses the \textit{GSPP} of the form \eqref{eq11} with  $A$ is symmetric and  $B=C$ is a Toeplitz matrix (see Section \ref{WTRLS_problem}).

	Perturbation theory is extensively used in numerical analysis to examine the sensitivity of numerical techniques and the error analysis of a computed solution \cite{Higham2002}. Condition numbers (\textit{CNs}) and backward errors are the two most important tools in perturbation theory. For a given problem, the \textit{CN} is the measurement of the worst-case sensitivity of a numerical solution to a tiny perturbation in the input data, whereas backward error reveals the stability of any numerical approach. Combined with the backward error, \textit{CNs} can provide a first-order estimate of the forward error of an approximated solution. 
	
	Rice in \cite{Rice1966} presented the classical theory of \textit{CNs}. It essentially deals with the normwise condition number (\textit{NCN}) by employing norms to measure both the input perturbation and error in the output data. A notable drawback associated with \textit{NCN} lies in its inability to capture the inherent structure of badly scaled or sparse input data. Consequently,  the \textit{NCN} occasionally overestimates the true conditioning of the numerical solution.   As a remedy for this, the mixed condition number (\textit{MCN}) and componentwise condition number (\textit{CCN}) have seen a growing interest in the literature \cite{Gohberg1993, rohn1989new, skeel1979scaling}. The former measures perturbation in input data componentwise and output data error by norms, while the latter measures both input perturbation and output data error componentwise.
 
	 {Perturbation theory and \textit{CNs} for the \textit{GSPP} \eqref{eq11} have been widely studied in the literature. A brief review of the literature work of the \textit{CNs} for the \textit{GSPP} \eqref{eq11}  is as follows.
  \citet{conditionwang} have analyzed the  normwise \textit{CN} for the solution $z=[x^{\T},y^{\T}]^{\T}$ for the KKT system, i.e., the \textit{GSPP} \eqref{eq11} with $A=A^{\T},$ $B=C$ and $D={\bf 0}$.
     In \cite{Xiangcnd2007}, authors have  discussed perturbation bounds  for the \textit{GSPP} when $B=C,$ 
     and $D={\bf 0},$ and  have derived closed formulae for the \textit{NCN}, \textit{MCN}, and \textit{CCN} of the solutions $z=[x^{\T},y^{\T}]^{\T}$ and the individual solution components $x$ and $y.$
      The \textit{NCN} and perturbation bounds have been investigated in \cite{weiwei2009}  for the solution  $z = [x^{\T}, y^{\T}]^{\T}$  of the \textit{GSPP} \eqref{eq11}, with the conditions $ B = C $ and  $ D\neq \mathbf{0}.$
    Later, \citet{meng2019condition} studied the \textit{MCN} and \textit{CCN} for $z = [x^{\T}, y^{\T}]^{\T}.$ Additionally, they explored the  \textit{NCN}, \textit{MCN}, and \textit{CCN} for the individual solution components $x$ and $y.$ 
   Recently, new perturbation bounds have been derived for the \textit{GSPP} \eqref{eq11} under the condition \( B \neq C \), without imposing any special structure on the $ A$ and $D$ \cite{Wei-Wei2017}.}
   
In many applications,  blocks of the coefficient matrix $\M$ of the system (\ref{eq11}) exhibit linear structures (for example, symmetric, Toeplitz or symmetric-Toeplitz)  {\cite{SymbolCirculant, Elman2005, SS2015, SYMTEOP}}. 
 Therefore, it is reasonable to ask: how sensitive is the solution when structure-preserving perturbations are introduced to the coefficient matrix of  \textit{GSPPs}? To address the aforementioned query, we explore the notion of structured \textit{CNs} by restricting perturbations that preserve the structures inherent in the block matrices of $\M.$ 

Furthermore, in many instances, $x$ and $y$ represent two distinct physical entities; for example, in the Stokes equation{ ,} $x$ denotes the velocity vector, and $y$ signifies the scalar pressure field  { \cite{Elman2005}}. Therefore, it is important to assess their individual conditioning properties. 
To accomplish this, we propose a general framework for assessing the conditioning of $x,$ $y,$ $z=[x^{\T},\, y^{\T}]^{\T}$ and each solution component. In the proposed general framework, we consider the structured \textit{CNs} of a linear function  $\mathbf{L}[x^{\T},\, y^{\T}]^{\T}$ of the solution to \textit{GSPP} \eqref{eq11}, where $\mathbf{L}\in \R^{k\times (m+n)}.$
 The matrix $\mathbf{L}$ serves as a pivotal tool for the purpose of selecting solution components. For example, $(i)$ $\mathbf{L}=I_{m+n}$ gives the \textit{CNs} for $[x^{\T},\, y^{\T}]^{\T},$ $(ii)$ $\mathbf{L}=\bmatrix{I_m & {\bf 0}}$ gives the \textit{CNs} for $x,$  {and} $(iii)$ $\mathbf{L}=\bmatrix{ {\bf 0} & I_n}$ gives the \textit{CNs} for $y$.  {Here, $I_m$ stands for the identity matrix of order $m$.}

\noin The key contributions of this paper are summarized as follows:
	\begin{itemize}
	\item We study the \textit{NCN}, \textit{MCN}, and \textit{CCN} for the linear function $\mathbf{L}[x^{\T},\, y^{\T}]^{\T},$ which in turn provides a general framework, enabling us to derive \textit{CNs} for the solutions $[x^{\T},\, y^{\T}]^{\T},$ $x$, $y,$ and each solution component.
 \item We investigate unstructured \textit{CNs} for $\mathbf{L}[x^{\T},\, y^{\T}]^{\T}$ by considering $B=C$ and then, the structured \textit{CNs} when the (1,1)--block $A$ is symmetric and (1,2)--block $B$ is Toeplitz. We derive explicit formulae for both unstructured and structured \textit{CNs}.  For appropriate choices of $\mathbf{L},$ we have shown that our derived unstructured \textit{CNs} formulae generalize the results given in literature \cite{ meng2019condition, weiwei2009}. 
		\item By considering linear structure on the block matrices $A$ and $D$ with $B\neq C,$ we provide compact formulae of the structured \textit{NCN}, \textit{MCN}, and \textit{CCN} for the linear function $\mathbf{L}[x^{\T},\, y^{\T}]^{\T}$ of the \textit{GSPP}  (\ref{eq11}). 
		\item Utilizing the structured \textit{CNs} formulae, we derive the structured \textit{CNs} for the \textit{WTRLS} problem and generalize some of the previous structured \textit{CNs}  {formulae} for the Tikhonov regularization problem.  This shows the generic nature of our obtained results. 
  
		\item Numerical experiments demonstrate that the obtained structured \textit{CNs} offer sharper bounds to the actual relative errors than their unstructured counterparts.
	\end{itemize}
	
	The organization of this paper is as follows. Section \ref{sec2}  {discusses } notation and preliminary results about \textit{CNs}. In Section \ref{sec3} and \ref{sec4}, we investigate unstructured and structured  \textit{NCN}, \textit{MCN}, and \textit{CCN} for the linear function $\mathbf{L}[x^{\T},\, y^{\T}]^{\T}$  of the solution of the \textit{GSPP}. Furthermore, an application of our obtained structured \textit{CNs} is provided in Section \ref{WTRLS_problem}  {for} \textit{WTRLS}  problems.  { Additionally, these \textit{CNs} are used to retrieve some prior found results for Tikhonov regularization problems}. In Section \ref{sec7},  numerical  {experiments} are carried out to demonstrate the effectiveness of the proposed structured \textit{CNs}. Section \ref{sec8}  {presents} some concluding remarks.

	\section{Notation and preliminaries}\label{sec2}
	In this section, we define some notation and review some well-known results, which play a crucial role in constructing the main findings of this paper. 
\subsection{Notation}
 {Let $\R^{m\times n}$  be the set of all $m\times n$ real matrices, and} { $\|\cdot \|_2$, $\|\cdot \|_{\infty},$ and $\|\cdot \|_F$ stand for the Euclidean norm$/$matrix $2$-norm, infinity norm, and Frobenius norm, respectively.}  For $x=[x_1, x_2,\ldots, x_n]^{\T}\in \R^{n},$ we denote $\mathbf{D}_x\in \R^{n\times n}$ as the diagonal matrix with $\mathbf{D}_x(i,i)=x_i.$  The symbol $A^{\dagger}$ denotes the Moore-Penrose inverse of $A.$   Following \cite{DIAO2007, Hanyu2014}, the entrywise division of any two vectors $x,y\in \R^n$  is defined as  $\frac{x}{y}:=[\frac{x_i}{y_i}],$ where  $x_i/0=0$  whenever $x_i=0$ and infinity otherwise.  For any matrix $A=[a_{ij}]\in \R^{m\times n},$ we set $|A|{ :}=[|a_{ij}|],$ where $|a_{ij}|$ denotes the absolute value of $a_{ij}.$   For any two matrices $A,B\in\R^{m\times n},$  the notation $A~
\leq B$ represents $a_{ij}~\leq b_{ij}$ for all $1\leq i\leq m$ and $1\leq j\leq n$.  For the matrix $A=[a_1, a_2,\ldots, a_n]\in \R^{m\times n},$ where $a_i\in \R^{m},$ $i=1,2,\ldots, n,$ the linear operator $\vec:\R^{m\times n}\mapsto\R^{mn}$ is defined by $\vec(A):=[a_1^{\T}, a_2^{\T},\ldots, a_n^{\T}]^{\T}$.
 The $\vec$ operator satisfies  $\|\vec(A)\|_{\infty}=\|A\|_F.$ The Kronecker product   {\cite{kronecker1981}} of two matrices $X\in \R^{m\times n}$ and $Y\in \R^{p\times q}$ is defined by $X\otimes Y{ :}=[x_{ij}Y]\in \R^{mp\times nq}$ and some of its important properties  are listed below \cite{kronecker1981,kronecker2004}:
	\begin{align}\label{eq13}
		&\vec(XCY)=(Y^{\T}\otimes X)\vec(C), \quad
		  |X\otimes Y|=|X|\otimes |Y|,
	\end{align}
	where $C\in \R^{n\times p}.$
 \subsection{Preliminaries}
	 {Throughout this paper, we assume that $A$ and $\M$ are nonsingular. We know that if  $A$ is nonsingular, then $\M$ is nonsingular if and only if its Schur complement $S=D-CA^{-1}B^{\T}$ is nonsingular \cite{Bai2013}
 and its  inverse is expressed as follows:
	\begin{align}\label{eq12}
		\M^{-1}=\bmatrix{A^{-1} +A^{-1}B^{\T}S^{-1}CA^{-1}&-A^{-1}B^{\T}S^{-1}\\ -S^{-1}CA^{-1} & S^{-1}}.
	\end{align}}
 
  {According to \cite{DIAO2007, Hanyu2014},} we define the following notations.
 The componentwise distance between two vectors $a$ and $b$  {in $\R^{p}$} is defined as 
	\begin{equation}
		d(a,b)=\Vr \frac{a-b}{b}\Vr_{\infty}=\max_{i=1,2,\ldots,p}\left\{\frac{|a_i-b_i|}{|b_i|}\right\}.
	\end{equation}
	Let $u\in \R^p$ and $\eta> 0,$ consider sets: $B_1(u,\eta)=\{x\in \R^p : \|x-u\|_2\leq \eta \|u\|_2\}$ and $B_2(u,\eta)=\{x\in \R^p : |x_i-u_i|\leq \eta|u_i|,\, i=1,\ldots,p\}.$ 
	
	With the above conventions, next, we present the definitions of \textit{NCN}, \textit{MCN}, and \textit{CCN} for a mapping $\bm{\varphi}:\R^p\mapsto \R^q$ as follows.

 \vone
	\begin{definition}\label{def21} {\cite{DIAO2007, Gohberg1993}}
		Let $\bm{\varphi}:\R^p\mapsto \R^q$ be a continuous mapping defined on an open set $\Omega_{\bm{\varphi}}\subseteq \R^p,$ and ${\bf 0}\neq u\in \Omega_{\bm{\varphi}}$ such that $\bm{\vp}(u)\neq {\bf 0}.$ 
		\begin{description}
			\item[$(i)$] The \textit{NCN} of $\bm{\vp}$ at $u$ is defined by 
			\begin{align*}
				\K(\bm{\vp}, u)=\lim_{\eta\rightarrow 0}\sup_{\underset{x\in B_1(u,\eta)}{x\neq u }}\frac{\|\bm{\vp}(x)-\bm{\vp}(u)\|_2/\|\bm{\vp}(u)\|_2}{\|x-u\|_2/\|u\|_2}.
			\end{align*}
			\item[$(ii)$] The \textit{MCN} of $\bm{\vp}$ at $u$ is defined by 
			\begin{align*}
				\Ms(\bm{\vp}, u)=\lim_{\eta\rightarrow 0}\sup_{\underset{x\in B_2(u,\eta)}{x\neq u }}\frac{\|\bm{\vp}(x)-\bm{\vp}(u)\|_{\infty}}{\|\bm{\vp}(u)\|_{\infty}} \frac{1}{d(x,u)}.
			\end{align*}
			\item[$(iii)$]  {Let} $\bm{\vp}(u)=[\bm{\vp}(u)_1,\ldots, \bm{\vp}(u)_q]^{\T}$  {be} such that $\bm{\vp}(u)_i \neq 0$ for $i=1,2, \ldots, q.$ Then, the \textit{CCN} of $\bm{\vp}$ at $u$ is defined by  \begin{align*}
				\C(\bm{\vp}, u)=\lim_{\eta\rightarrow 0}\sup_{\underset{x\in B_2(u,\eta)}{x\neq u }} \frac{d(\bm{\vp}(x),\bm{\vp}(u))}{d(x,u)}.
			\end{align*}
		\end{description}
	\end{definition}
 \begin{definition}\cite{frechet}
     Let $\bm{\varphi}:\R^p\mapsto \R^q$ be a mapping defined on an open set 
  $\Omega_{\bm{\varphi}}\subseteq \R^p.$ Then $\bm{\varphi}$ is said to be $Fr\acute{e}chet$ differentiable at $u\in \Omega_{\bm{\varphi}}$ if there exists a bounded linear operator $\d \bm{\vp}:\R^p\mapsto \R^q $ such that
  \begin{align*}
      \lim_{h\rightarrow {\bf 0}}\frac{\|\bm{\varphi}(u+h)-\bm{\varphi}(u)-\d\bm{\vp} h\|}{\|h\|}=0,
  \end{align*}
   {where $\|\cdot\|$ denotes any norm on $\R^{p}$ and $\R^q.$ }
 \end{definition}
 
 When $\bm{\varphi}$ is $Fr\acute{e}chet$ differentiable at $u,$ we denote the $Fr\acute{e}chet$ derivative at $u$ as $\d \bm{\vp}(u).$  The next lemma gives closed form expressions for the above three \textit{CNs} when the continuous mapping $\bm{\vp}$ is $Fr\acute{e}chet$ differentiable.
	\vspace{2mm}
 \begin{lemma}\label{lm21} {\cite{DIAO2007, Gohberg1993}}
		Under the same hypothesis as of Definition \ref{def21}, when $\bm{\vp}$ is $Fr\acute{e}chet$ differentiable at $u$, we have \begin{align*}
			\K(\bm{\vp}; u)=\frac{\|\d \bm{\vp}(u)\|_2 \|u\|_2}{\|\bm{\vp}(u)\|_2}, \quad \Ms(\bm{\vp}; u)=\frac{\||\d \bm{\vp}(u)||u|\|_{\infty}}{\|\bm{\vp}(u)\|_{\infty}}, \quad \C(\bm{\vp}; u)=\left\| \frac{|\d \bm{\vp}(u)||u|}{|\bm{\vp}(u)|}\right\|_{\infty},
		\end{align*}
		where $\d \bm{\vp}(u)$ denotes the $Fr\acute{e}chet$ derivative of $\bm{\vp}$ at $u.$
	\end{lemma}
\vspace{1mm}

First, consider the case when $B=C,$ i.e., the following \textit{GSPP} 
\begin{align}\label{Eq22}
		\M\bmatrix{x \\ y}:=	\bmatrix{ A  &B^{\T}\\ B & D}\bmatrix{x \\ y}=\bmatrix{f \\ g}:={\bf d},
	\end{align} 
and let $\D A,$ $\D B, \D D, \D f$ and $\D g$ be the perturbations in $A, B, D, f$ and $C,$ respectively. Then,  {we have the following  perturbed problem of \eqref{Eq22}:  }
\begin{align}\label{EQ22}
	(\M+\D \M)\bmatrix{x+\D x \\ y+\D y}=	\bmatrix{A+\D A & (B+\D B)^{\T}\\ B+\D B & D+\D D}\bmatrix{x+\D x\\ y+\D y}=\bmatrix{f+\D f\\ g+\D g},
	\end{align}
 {which} has the unique solution $\bmatrix{x+\D x\\ y+\D y}$  when $\|\M^{-1}\|_2\|\D \M\|_2<  1.$  Now, from \eqref{EQ22} omitting the higher order term, we obtain
 \begin{align}\label{EQ23}
		\bmatrix{\D x\\ \D y}\approx \M^{-1} \bmatrix{\D f\\ \D g}-	
		\M^{-1} \bmatrix{\D A & \D B^{\T}\\ \D B & \D D}\bmatrix{x \\ y}.
	\end{align}
 
 Using the properties in \eqref{eq13}, we have the following important lemma.
 \vspace{1.5mm}
 \begin{lemma}\label{LM22}
 Let $\bmatrix{ x\\ y}$ and $\bmatrix{x+ \D x\\ y+ \D y}$ be the unique solutions of the \textit{GSPP} \eqref{Eq22}
and \eqref{EQ22}, respectively. Then we have the following perturbation expression:
\begin{eqnarray*}
		\bmatrix{\D x\\ \D y}&\approx -\M^{-1}\bmatrix{\mathcal{R} & -I_{m+n}}\bmatrix{\vec(\D A) \\ \vec(\D B)\\ \vec(\D D)\\
			\D f \\ \D g
		},
	\end{eqnarray*}
	where 
 \begin{eqnarray}\label{EQ24}
\mathcal{R}=\bmatrix{x^{\T}\otimes I_{m} & I_m\otimes y^{\T}  & {\bf 0}\\
			{\bf 0}& x^{\T}\otimes I_{n} & y^{\T}\otimes I_{n}}.
	\end{eqnarray}
  \end{lemma}
  \begin{proof}
  { The proof} follows from \eqref{EQ23} and using the properties in \eqref{eq13}.  
  \end{proof}

Denote $\mathbf{H}=\bmatrix{A& {\bf 0}\\B &D} $ and $\D \mathbf{H}=\bmatrix{\D A& {\bf 0}\\ \D B &\D D}. $   \citet{weiwei2009} investigated unstructured  \textit{NCN} and \citet{meng2019condition}  studied unstructured \textit{MCN} and \textit{NCN} for the solution $[x^{\T}, \, y^{\T}]^{\T}$ to the  \textit{GSPP} \eqref{eq11} when $B=C,$ which are given as follows:
\begin{eqnarray}\label{EQ25}
 \K^u([x^{\T},\,  y^{\T}]^{\T}) &:=&\lim_{\eta\rightarrow 0}\sup\left\{\frac{\|[\Delta x^{\T},\, \Delta y^{\T}]^{\T} \|_2}{\eta\|[x^{\T},\, y^{\T}]^{\T}\|_2}: \left\|\bmatrix{\Delta \mathbf{H} &\D \textbf{d}}\right\|_F\leq \eta \left\|\bmatrix{\mathbf{ H}& \mathbf{d}}\right\|_F \right\}\\ \nonumber
&=&\frac{\left\|\M^{-1}\bmatrix{\mathcal{R} & -I_{m+n}}\right\|_2\left\|\bmatrix{\mathbf{ H}& \mathbf{d}}\right\|_F}{\left\|[x^{\T},\,  y^{\T}]^{\T}\right\|_2},\\
\label{EQ26}
    \Ms^u([x^{\T},\,  y^{\T}]^{\T})&:=&\lim_{\eta\rightarrow 0}\sup\left\{\frac{\|[\Delta x^{\T},\, \Delta y^{\T}]^{\T} \|_{\infty}}{\eta\|[x^{\T},\, y^{\T}]^{\T}\|_{\infty}} \left\vert\bmatrix{\Delta \mathbf{H} &\D \textbf{d}}\right\vert \, \leq \eta \left\vert \bmatrix{ \mathbf{H} & \textbf{d}}\right\vert \right\}\\  \nonumber
  &:=&\left\| |\M^{-1}\mathcal{R}|\bmatrix{\vec(|A|)\\ \vec(|B|)\\ \vec(|D|)}+|\M^{-1}|\bmatrix{|f|\\|g|} \right\|_{\infty}{\Bigg /}\left\|\bmatrix{x\\ y}\right\|_{\infty},\\
  \label{EQ27}
  \C^u([x^{\T},\,  y^{\T}]^{\T})&:=&\lim_{\eta\rightarrow 0}\sup\left\{\frac{1}{\eta}\left\|\frac{[\Delta x^{\T},\, \Delta y^{\T}]^{\T}}{[x^{\T},\, y^{\T}]^{\T}}\right\|_{\infty}:  \left\vert\bmatrix{\Delta \mathbf{H} &\D \textbf{d}}\right\vert \, \leq \eta \left\vert \bmatrix{ \mathbf{H} & \textbf{d}}\right\vert\right\}\\  \nonumber
  &=&\Bigg\| \mathbf{D}^{\dagger}_{[x^{\T},\, y^{\T}]^{\T}}|\M^{-1}\mathcal{R}|\bmatrix{\vec(|A|)\\ \vec(|B|)\\ \vec(|D|)}+ \mathbf{D}^{\dagger}_{[x^{\T},\, y^{\T}]^{\T}} |\M^{-1}|\bmatrix{|f|\\|g|}\Bigg\|_{\infty},
	\end{eqnarray}
where $\mathcal{R}$ is defined as in \eqref{EQ24}.

\vspace{1.5mm}
In the next section, we consider the unstructured and structured \textit{CNs} for a linear function of the solution of the \textit{GSPP} \eqref{Eq22}.
\section{ \textit{CNs} for a linear function of the solution to the \textit{GSPP} when $B=C$}\label{sec3}
In this section, we derive the compact \textit{NCN}, \textit{MCN}, and \textit{CCN} formulae for a linear function of the solution to the \textit{GSPP} \eqref{eq11} when $B=C,$ under both unstructured and structured perturbations. Additionally, comparisons between unstructured and structured \textit{CNs} are also provided.
\subsection{Unstructured \textit{CNs} forlumae}
In this subsection, we consider the unstructured \textit{CNs} for the linear function $\mathbf{L}[x^{\T},\, y^{\T}]^{\T},$ where $\mathbf{L}\in \R^{k\times (m+n)}$ and derive their explicit formulae.
In the following, we define unstructured \textit{NCN}, \textit{MCN}, and \textit{CCN} for the linear function $\mathbf{L}[x^{\T},\, y^{\T}]^{\T}.$ Throughout the paper, we assume $[x^{\T},\, y^{\T}]^{\T}\neq {\bf 0}$ for \textit{MCN} and $x_i\neq 0\, (i=1,\ldots, m)$ and $y_i\neq 0\, (i=1,\ldots, n)$ for \textit{CCN}.
\begin{definition}\label{def31}
    Let $[x^{\T},\, y^{\T}]^{\T}$ and { $\bmatrix{(x+\D x)^{\T}, \, (y+\D y)^{\T}}^{\T}$} be the unique solutions of  \textit{GSPPs} \eqref{Eq22} and \eqref{EQ22}, respectively, and $\mathbf{L}\in \R^{k\times {(m+n)}}.$ Then we define the unstructured \textit{NCN}, \textit{MCN}, and \textit{CCN} for the linear function $\mathbf{L}[x^{\T},\, y^{\T}]^{\T},$ respectively, as follows:
\begin{align*}
 &\K(\mathbf{L}[x^{\T},\,  y^{\T}]^{\T}):=\lim_{\eta\rightarrow 0}\sup\left\{\frac{\|\mathbf{L}[\Delta x^{\T},\, \Delta y^{\T}]^{\T} \|_2}{\eta\|\mathbf{L}[x^{\T},\, y^{\T}]^{\T}\|_2}: \left\|\bmatrix{\Delta \mathbf{H} &\D \mathbf{d}}\right\|_F\leq \eta \left\|\bmatrix{\mathbf{ H}& \mathbf{d}}\right\|_F \right\},\\
 & \Ms(\mathbf{L}[x^{\T},\,  y^{\T}]^{\T}):=\lim_{\eta\rightarrow 0}\sup\left\{\frac{\|\mathbf{L}[\Delta x^{\T},\, \Delta y^{\T}]^{\T} \|_{\infty}}{\eta\|\mathbf{L}[x^{\T},\, y^{\T}]^{\T}\|_{\infty}}: \left\vert\bmatrix{\Delta \mathbf{H} &\D \mathbf{d}}\right\vert \, \leq \eta \left\vert \bmatrix{ \mathbf{H} & \mathbf{d}}\right\vert \right\},\\
  &\C(\mathbf{L}[x^{\T},\,  y^{\T}]^{\T}):=\lim_{\eta\rightarrow 0}\sup\left\{\frac{1}{\eta}\left\|\frac{\mathbf{L}[\Delta x^{\T},\, \Delta y^{\T}]^{\T}}{\mathbf{L}[x^{\T},\, y^{\T}]^{\T}}\right\|_{\infty}:  \left\vert\bmatrix{\Delta \mathbf{H} &\D \mathbf{d}}\right\vert \, \leq \eta \left\vert \bmatrix{ \mathbf{H} & \mathbf{d}}\right\vert \right\}.
 \end{align*}
 \end{definition}
 
 Note that when $\mathbf{L}=I_{m+n},$  the above definitions reduces to \eqref{EQ25}--\eqref{EQ27}. For using Lemma \ref{lm21}, we construct the mapping $\bm{ \psi}: \R^{m^2+mn+n^2}\times \R^{m+n}\mapsto \R^{m+n}$ by 
 \begin{equation}\label{map1}
    \bm{ \psi}([\Omega^{\T}, f^{\T}, g^{\T}]^{\T}):= \mathbf{L}\bmatrix{x\\ y}=\mathbf{L}\M^{-1}\bmatrix{f\\ g},
 \end{equation}
 where $\Omega^{\T}=[\vec(A)^{\T}, \vec(B)^{\T}, \vec(D)^{\T}]^{\T}.$
 
 The following result is crucial for finding the \textit{CNs} formulae.
 \vspace{2mm}
 \begin{proposition}\label{PROP1}
      Let $\Omega^{\T}=[\vec(A)^{\T}, \vec(B)^{\T}, \vec(D)^{\T}]^{\T}.$ Then, for the map $\bm{ \psi}$ defined in \eqref{map1}, we have
 $\K(\mathbf{L}[x^{\T},\,  y^{\T}]^{\T})=\K(\bm{ \psi},[\Omega^{\T}, f^{\T}, g^{\T}]^{\T}),$  $\Ms(\mathbf{L}[x^{\T},\,  y^{\T}]^{\T})=\Ms(\bm{ \psi},[\Omega^{\T}, f^{\T}, g^{\T}]^{\T}),$ and  $\C(\mathbf{L}[x^{\T},\,  y^{\T}]^{\T})=\C(\bm{ \psi},[\Omega^{\T}, f^{\T}, g^{\T}]^{\T}).$ 
 \end{proposition}
 \vspace{2mm}
 \begin{proof}
      Let $\D\Omega^{\T}=[\vec(\D A)^{\T},  \vec(\D B)^{\T}, \vec(\D 
 D)^{\T}]^{\T}.$ Then, from \eqref{map1}, we obtain
 \begin{align}\label{pr:32}
\nonumber&\bm{\psi}\left([(\Omega^{\T}+\D\Omega^{\T}), f^{\T}+\D f^{\T},g^{\T}+\D g^{\T}]^{\T}\right)-\bm{\psi}\left([\Omega^{\T},\,  f^{\T},\, g^{\T}]^{\T}\right)\\
&=\mathbf{L}\bmatrix{x+\D x\\ y+\D y}- \mathbf{L}\bmatrix{ x\\  y}= \mathbf{L}\bmatrix{\D x\\ \D y}.
 \end{align}
 Now, in Definition \ref{def31}, substituting \eqref{pr:32} and $\bm{\psi}\left([\Omega^{\T},\,  f^{\T},\, g^{\T}]^{\T}\right)=\mathbf{L}\bmatrix{ x^{\T},\, y^{\T}}^{\T},$ and consequently from Definition \ref{def21},
 the proof follows.
  \end{proof}
 
Since the $Fr\acute{e}chet$ derivative of $\bm{ \psi}$  has a pivotal role in estimating the \textit{CNs} in Definition \ref{def31}, it is essential to derive simple expressions for $\d\bm{ \psi}.$
 By applying Lemma \ref{LM22}, we obtain the following results for   $\d \bm{\psi}.$
 \begin{lemma}\label{lmm31}
     The map $\bm{\psi}$ defined above is  continuous and $Fr\acute{e}chet$ differentiable at $[\Omega^{\T}, f^{\T}, g^{\T}]^{\T}$ and its  $Fr\acute{e}chet$ derivative at $[\Omega^{\T}, f^{\T}, g^{\T}]^{\T}$ is given by $${\bf d}\bm{\psi}([\Omega^{\T}, f^{\T}, g^{\T}]^{\T})=-\mathbf{L}\M^{-1} \bmatrix{\mathcal{R} & -I_{m+n}}.$$
 \end{lemma}
 \begin{proof}
 Since $\M^{-1}$ is continuous in its elements, the linear map $\bm{\psi}$ is also continuous. 
 Let $\D\Omega^{\T}=[\vec(\D A)^{\T},  \vec(\D B)^{\T}, \vec(\D 
 D)^{\T}]^{\T}.$ Then $$\bm{\psi}\left([(\Omega^{\T}+\D\Omega^{\T}), f^{\T}+\D f^{\T},g^{\T}+\D g^{\T}]^{\T}\right)-\bm{\psi}\left([\Omega^{\T},\,  f^{\T},\, g^{\T}]^{\T}\right)= \mathbf{L}\bmatrix{\D x\\ \D y}.$$ Hence, the rest of the proof follows from the Lemma \ref{LM22}. 
 \end{proof}

   Applying Lemma \ref{lmm31}, we obtain the following closed formulae for the unstructured \textit{CNs} for the linear function $\mathbf{L}[x^{\T},\, y^{\T}]^{\T}.$

   \vspace{2mm}
 \begin{theorem}\label{Th31}
 Let $[x^{\T},\, y^{\T}]^{\T}$ be the unique solution of the \textit{GSPP} \eqref{Eq22}.  Then the unstructured \textit{NCN}, \textit{MCN}, and \textit{CCN} for the linear function $\mathbf{L}[x^{\T},\, y^{\T}]^{\T},$ respectively, are given by
     \begin{eqnarray*}
         &&\K(\mathbf{L}[x^{\T},\,  y^{\T}]^{\T})=\frac{\left\|\mathbf{L}\M^{-1}\bmatrix{\mathcal{R} & -I_{m+n}}\right\|_2\left\|\bmatrix{\mathbf{ H}& \mathbf{d}}\right\|_F}{\left\|\mathbf{L}[x^{\T},\,  y^{\T}]^{\T}\right\|_2},\\
         &&\Ms(\mathbf{L}[x^{\T},\,  y^{\T}]^{\T})=\frac{\left\| |\mathbf{L}\M^{-1}\mathcal{R}|\bmatrix{\vec(|A|)\\ \vec(|B|)\\ \vec(|D|)} + |\mathbf{L} \M^{-1}|\bmatrix{|f|\\|g|} \right\|_{\infty}}{\|\mathbf{L}[x^{\T},\, y^{\T}]^{\T}\|_{\infty}},\\
         &&\C(\mathbf{L}[x^{\T},\,  y^{\T}]^{\T})=\Bigg\| \mathbf{D}^{\dagger}_{\mathbf{L}[x^{\T},\, y^{\T}]^{\T}}|\mathbf{L}\M^{-1}\mathcal{R}|\bmatrix{\vec(|A|)\\ \vec(|B|)\\ \vec(|D|)} + \mathbf{D}^{\dagger}_{\mathbf{L}[x^{\T},\, y^{\T}]^{\T}} |\mathbf{L} \M^{-1}|\bmatrix{|f|\\|g|}\Bigg\|_{\infty}.
     \end{eqnarray*}
 \end{theorem}
 \begin{proof}
 Let $\Omega^{\T}=[\vec(A)^{\T}, \vec(B)^{\T}, \vec(D)^{\T}]^{\T}.$ Then{ ,} from Proposition \ref{PROP1} and applying the \textit{NCN} formula of Lemma \ref{lm21} for the map $\bm{\psi}{ ,}$ we obtain
 \begin{align}\label{TH:eq31}
			\K(\mathbf{L}[x^{\T},\, y^{\T}]^{\T})=\K(\bm{ \psi},[\Omega^{\T}, f^{\T}, g^{\T}]^{\T})&= \frac{\Vr \d \bm{\psi}\left(\Omega^{\T},\,  f^{\T},\, g^{\T}]^{\T}\right) \Vr_{2}\left\|[\Omega^{\T}, f^{\T}, g^{\T}]^{\T}\right\|_2}{\left\|\bm{\psi}\left(\Omega^{\T},\,  f^{\T},\, g^{\T}]^{\T}\right)\right\|_2}.
   \end{align}
 Now, substituting the $Fr\acute{e}chet$ derivative expression of $\bm{\psi}$  at $[\Omega^{\T}, f^{\T}, g^{\T}]^{\T}$ provided in Lemma \ref{lmm31} in \eqref{TH:eq31}, we get \begin{align*}
			\K(\mathbf{L}[x^{\T},\, y^{\T}]^{\T})&=\frac{\Vr\mathbf{L}\M^{-1}\bmatrix{\mathcal{R}  & -I_{m+n}}\Vr_{2}\left\|\bmatrix{ \mathbf{H} & \mathbf{d}}\right\|_F}{\left\|\mathbf{L}[x^{\T},\, y^{\T}]^{\T}\right\|_2}.
		\end{align*} 
  Similarly, applying the \textit{MCN} formula provided in Lemma \ref{lm21} for $\bm{\psi}$, we get 
  \begin{align}\label{TH:eq32}
      \Ms(\mathbf{L}[x^{\T},\, y^{\T}]^{\T})=\Ms(\bm{ \psi},[\Omega^{\T}, f^{\T}, g^{\T}]^{\T})&=\frac{\left\||\d\bm{\psi}\left([\Omega^{\T},\,  f^{\T},\, g^{\T}]^{\T}\right) |\left| [\Omega^{\T}, f^{\T}, g^{\T}]^{\T}\right|\right\|_{\infty}}{\left\|\bm{\psi}\left([\Omega^{\T},\,  f^{\T},\, g^{\T}]^{\T}\right)\right\|_{\infty}}.
  \end{align}
Substituting the  $Fr\acute{e}chet$ derivative expression provided in Lemma \ref{lmm31} in \eqref{TH:eq32},  {we obtain}
  \begin{align*}
			\Ms(\mathbf{L}[x^{\T},\, y^{\T}]^{\T})&=\frac{\left\|\vr\mathbf{L}\M^{-1}\bmatrix{\mathcal{R} & -I_{m+n}}\vr\left |[\Omega^{\T}, f^{\T}, g^{\T}]^{\T}\right|\right\|_{\infty}}{\left\|\mathbf{L}[x^{\T},\, y^{\T}]^{\T}\right\|_{\infty}}\\
			&=\frac{\Bigg\| |\mathbf{L}\M^{-1}\mathcal{R}|\bmatrix{\vec(|A|)\\ \vec(|B|)\\ \vec(|D|)}+|\mathbf{L}\M^{-1}|\bmatrix{|f|\\|g|}\Bigg\|_{\infty}}{\Vr\mathbf{L}[x^{\T},\, y^{\T}]^{\T}\Vr_{\infty}}.
		\end{align*}
The rest of the proof follows in a similar way.
\end{proof}
 \vspace{1.5mm}
 \begin{remark}
     If we consider $\mathbf{L}=I_{m+n}, $ then the formulae of $\K(\mathbf{L}[x^{\T},\,  y^{\T}]^{\T}), \Ms(\mathbf{L}[x^{\T},\,  y^{\T}]^{\T})$ and $\C(\mathbf{L}[x^{\T},\,  y^{\T}]^{\T})$ reduces to unstructured \textit{CNs} $\K^u([x^{\T},\,  y^{\T}]^{\T}), \Ms^u([x^{\T},\,  y^{\T}]^{\T})$ and  $\C^u([x^{\T},\,  y^{\T}]^{\T})$ given in \eqref{EQ25}--\eqref{EQ27}, respectively. Moreover, if we choose $\mathbf{L}=\bmatrix{I_m & {\bf 0}}$ and $\mathbf{L}=\bmatrix{ {\bf 0} & I_n},$ and after some easy calculations, we can recover the unstructured \textit{CNs} formulae of \cite{meng2019condition} for $x$ and $y,$ respectively. 
 \end{remark}
 \subsection{Structured \textit{CNs} when $A$ is symmetric and $B=C$ is  Toeplitz}\label{subsec32}
 In this section, we consider the structured \textit{NCN}, \textit{MCN}, and \textit{CCN} of the \textit{GSPP} (\ref{Eq22}) with $A= A^{\T}$  and $B\in \R^{n\times m}$  is a  Toeplitz matrix. We denote $\SY_m$ and  $\mathcal{T}^{n\times m}$ as set of all $m\times m$ symmetric matrices and $n\times m$  Toeplitz matrices, respectively.

Now, we recall the definition of Toeplitz matrices and derive a few important lemmas.
\vspace{1.5mm}
\begin{definition}\cite{Johnson}
   A matrix $T=[t_{ij}]\in \R^{n\times m} $ is called a Toeplitz matrix if there exists $${\bf t}=[t_{-n+1},  \ldots, t_{-1}, t_0, t_1, \ldots, t_{m-1}]^{\T} \in \R^{m+n-1}$$ such that $t_{ij}=t_{j-i},$ for all $1\leq i\leq n$ and $1\leq j\leq m.$

   The generator vector ${\bf t}$ for $T$ is denoted by $\vec_{\mathcal{T}}(T).$ Moreover, for ${\bf t}\in \R^{m+n-1},$ 
 ${\tt Toep}(\bf t)$ is denoted the 
 corresponding generated Toeplitz matrix.
\end{definition}

As $\dim(\mathcal{T}^{n\times m})=m+n-1,$ consider the basis $\left\{\mathcal{J}_{i}\right\}_{i=-n+1}^{m-1}$  for $\mathcal{T}^{n\times m}$ defined as
	$$\mathcal{J}_{i}=\left\{ \begin{array}{lcl}
		{\tt Toep}([(e_{n-i}^{(n)})^{\T},\, {\bf 0}]^{\T}) & \mbox{for} & i=-n+1, \ldots, -1, 0,\\
  {\tt Toep}( [{\bf 0}, \,(e_i^{(m)})^{\T}]^{\T})& \mbox{for} & i=  1, \ldots, m-1,
	\end{array}\right.$$\\
	where  $e_i^{(m)}$ is the $i$-th column of  $I_m.$  Moreover, construct the diagonal matrix $\mathfrak{D}_{\mathcal{T}_{nm}}\in \R^{(m+n-1)\times (m+n-1)}$   with  $\mathfrak{D}_{\mathcal{T}_{nm}}(j,j)=\bm{a}_j,$ where  $$\bm{a}=[1, \sqrt{2}, \ldots, \sqrt{n-1}, \sqrt{\min\{m,n\}}, \sqrt{m-1},\ldots, \sqrt{2}, 1]^{\T}\in \R^{m+n-1}$$ such that $\|T\|_F=\|\mathfrak{D}_{\mathcal{T}_{nm}}\vec_{\mathcal{T}}(T)\|_2.$ 
 \vone
 \begin{lemma}\label{Lmm31}
    Let $T\in \mathcal{T}^{n\times m},$ then $\vec(T)=\bm{\Phi}_{\mathcal{T}_{nm}}\vec_{\mathcal{T}}(T),$ 
    where  $$\bm{\Phi}_{\mathcal{T}_{nm}}=\bmatrix{\vec(\mathcal{J}_{-n+1}), \ldots, \vec(\mathcal{J}_{m-1})}\in \R^{mn\times (m+n-1)}.$$
 \end{lemma}
 \begin{proof}
  Assume that $\vec_{\mathcal{T}}(T)=[t_{-n+1}, \ldots t_0,\ldots, t_{m-1}]^{\T},$ then $$T=\sum_{i=-n+1}^{m-1}t_i\mathcal{J}_{i}\Longleftrightarrow \vec(T)=\bm{\Phi}_{\mathcal{T}_{nm}}\vec_{\mathcal{T}}(T). $$ 
     Hence, the proof follows.   
 \end{proof}
     \vone
     
     Let  $A\in \SY_{m},$ then $A=A^{\T}.$ 
Moreover, we have	$\dim(\SY_m)=\frac{m(m+1)}{2}=:\bm{p}.$  We denote the generator vector for $A$ as 
 \begin{equation*}
     \vec_{\SY}(A):=[a_{11},\ldots,a_{1m},a_{22},\ldots,a_{2m},\ldots,a_{(m-1)(m-1)},a_{(m-1)m}, a_{mm}]^{\T}\in \R^{\bm{p}}.
 \end{equation*}
	Consider the basis $\left\{S_{ij}^{(m)}\right\}$  for $\SY_{m}$ defined as
	$$S^{(m)}_{ij}=\left\{ \begin{array}{lcl}
		e_i^{(m)}(e_j^{(m)})^{\T}+(e_j^{(m)}e_i^{(m)})^{\T} & \mbox{for} & i\neq j, \\ e_i^{(m)}(e_i^{(m)})^{\T}& \mbox{for} & i=j,
	\end{array}\right.$$\\
	where $1\leq i\leq j\leq m.$ Then, we have the following immediate result for vec-structure of $A.$ 
 \vspace{2mm}
 \begin{lemma}\label{LM33}
    Let $A\in \SY_{m},$ then $\vec(A)=\bm{\Phi}_{\mathcal{S}_m}\, \vec_{\mathcal{S}}(A),  $
    where $\bm{\Phi}_{\mathcal{S}_m}\in \R^{m^2\times {\bm p}}$ is given by
    $$\bm{\Phi}_{\mathcal{S}_m}=\bmatrix{\vec(S^{(m)}_{11})& \cdots& \vec(S^{(m)}_{1m}) &\vec(S^{(m)}_{22})& \cdots &\vec(S^{(m)}_{2m})&\cdots&\vec(S^{(m)}_{(m-1)m}) &\vec(S^{(m)}_{mm})}.$$
 \end{lemma}
 \begin{proof}
 The proof follows by using the similar proof method of Lemma  \ref{Lmm31}. 
 \end{proof}
 
 \vspace{2mm}
We construct the diagonal matrix $\mathfrak{D}_{\SY_{m}}\in \R^{\bm{p}\times \bm{p}},$ where
 $$\left\{ \begin{array}{lcl}
		\mathfrak{D}_{\SY_{m}}(j,j)=1 & \mbox{for} & {  j=\frac{(2m-(i-2))(i-1)}{2}+1},\, i=1,2,\ldots,m, \\ \mathfrak{D}_{\SY_{m}}(j,j)=\sqrt{2}& \mbox{for} & \text{otherwise}.
	\end{array}\right.$$
 This matrix satisfies the property $\|A\|_F=\|\mathfrak{D}_{\SY_{m}}\vec_{\mathcal{S}}(A)\|_2.$

  Consider the set  $$\E=\left\{\mathbf{H}=\bmatrix{A& {\bf 0}\\B &D} :  \, A\in \SY_{m}, B\in \mathcal{T}^{n\times m}, D\in \R^{n\times n}\right\},$$
   {and let $\Delta\mathbf{H}= \bmatrix{\Delta A& {\bf 0}\\ \Delta B &\Delta D} \in \E,$ i.e., $\D A \in \SY_m, \, \D B\in \mathcal{T}^{n\times m},$ and $\D D\in \R^{n\times n}.$}
  
  Next, we define the structured \textit{CNs} for the solution of the \textit{GSPP} \eqref{Eq22}.
\vspace{1mm}
\begin{definition}\label{Def32}
   Let $[x^{\T},\, y^{\T}]^{\T}$ and { $\bmatrix{(x+\D x)^{\T}, \, (y+\D y)^{\T}}^{\T}$} be the unique solutions of \textit{GSPPs} \eqref{Eq22} and \eqref{EQ22}, respectively, with the structure $\E$ and  $\mathbf{L}\in \R^{k\times {(m+n)}}.$  Then, the structured \textit{NCN}, \textit{MCN}, and \textit{CCN} for the linear function $\mathbf{L}[x^{\T},\, y^{\T}]^{\T}$ are defined as follows:
     \begin{align*}
         & \K(\mathbf{L}[x^{\T},\, y^{\T}]^{\T}; \, \E):=\lim_{\eta\rightarrow 0}\sup\Bigg\{\frac{\|{\mathbf{L}}[\Delta x^{\T},\, \Delta y^{\T}]^{\T} \|_2}{\eta\|\mathbf{L}[x^{\T},\, y^{\T}]^{\T}\|_2}: \left\|\bmatrix{\Delta \mathbf{H} &\D \bf{d}}\right\|_F\leq \eta \left\|\bmatrix{\mathbf{ H}& \mathbf{d}}\right\|_F, \D\mathbf{H}\in \mathcal{E} \Bigg\},\\
         &\Ms(\mathbf{L}[x^{\T},\, y^{\T}]^{\T};\,
         \E):=\lim_{\eta\rightarrow 0}\sup\Bigg\{\frac{\|{\mathbf{L}}[\Delta x^{\T},\, \Delta y^{\T}]^{\T} \|_{\infty}}{\eta\|\mathbf{L}[x^{\T},\, y^{\T}]^{\T}\|_{\infty}}: \left\vert\bmatrix{\Delta \mathbf{H} &\D \bf{d}}\right\vert \, \leq \eta \left\vert \bmatrix{ \mathbf{H} & \bf{d}}\right\vert, \D\mathbf{H}\in \mathcal{E} \Bigg\},\\ 
         &\C(\mathbf{L}[x^{\T},\, y^{\T}]^{\T};\,\E):=\lim_{\eta\rightarrow 0}\sup\Bigg\{\frac{1}{\eta}\left\|\frac{\mathbf{L}[\Delta x^{\T},\, \Delta y^{\T}]^{\T}}{\mathbf{L}[x^{\T},\, y^{\T}]^{\T}}\right\|_{\infty}:  \left\vert\bmatrix{\Delta \mathbf{H} &\D \bf{d}}\right\vert \, \leq \eta \left\vert \bmatrix{ \mathbf{H} & \bf{d}}\right\vert, \D\mathbf{H}\in \mathcal{E}\Bigg\}.
     \end{align*}
 \end{definition}

	To find the structured \textit{CNs} formulae by employing Lemma \ref{lm21}, we define the following mapping
	\begin{align}\label{eq69}
		&\bm{\zeta}:\, \R^{\bm l}\times \R^m\times \R^n \mapsto \R^{m+n} \quad \mbox{by} \\ \nonumber
		&\bm{\zeta}\left([\mathfrak{D}_{\E}{ \bm w}^{\T},\,  f^{\T},\, g^{\T}]^{\T}\right)=\textbf{L}\bmatrix{x\\ y}=\textbf{L}\M^{-1}\bmatrix{f\\g},
	\end{align}
	where ${\bm l}=n^2+\bm{p}+m+n-1,$ $\bm{w}=\bmatrix{\vec_{\mathcal{S}}(A)\\ \vec_{\mathcal{T}}(B)\\ \vec(D)}$ and $\mathfrak{D}_{\E}=\bmatrix{\mathfrak{D}_{\SY_{m}} &\textbf{0}& {\bf 0}\\  {\bf 0} & \mathfrak{D}_{\mathcal{T}_{nm}} & {\bf 0}\\\textbf{0} &  {\bf 0} & I_{n^2}}.$
	
	In the next lemma, we provide  $Fr\acute{e}chet$ derivative of the map $\bm{\bm{\zeta}}$ at  $\bmatrix{\mathfrak{D}_{\E} \bm{w}^{\T},\,  f^{\T},\, g^{\T}}^{\T} $. 
	\vspace{1.5mm}
 \begin{lemma}\label{lmm33}
		The mapping $\bm{\zeta}$ defined in (\ref{eq69}) is continuously $Fr\acute{e}chet$ differentiable at $\bmatrix{\mathfrak{D}_{\E} { \bm w}^{\T},\,  f^{\T},\, g^{\T}}^{\T} $ and the $Fr\acute{e}chet$ derivative is given by
		$$\d \bm{\zeta}\left([\mathfrak{D}_{\E}{ \bm w}^{\T},\,  f^{\T},\, g^{\T}]^{\T}\right)=-\mathbf{L}\M^{-1}\bmatrix{\mathcal{R}\bm{\Phi}_{\mathbf{\E}} \mathfrak{D}_{\E}^{-1} & -I_{m+n}},$$
		where $\bm{\Phi}_{\E}=\bmatrix{\bm{\Phi}_{\SY_{m}} &\bf{0}& {\bf 0}\\  {\bf 0} & \bm{\Phi}_{\mathcal{T}_{nm}} & {\bf 0}\\\bf{0} &  {\bf 0} & I_{n^2}}.$
	\end{lemma}
 \begin{proof}
  The continuity of the linear map $\bm{\zeta}$ follows from the continuity of $\M^{-1}.$  For the second part, let $\D \bm{w}=\bmatrix{\vec_{\mathcal{S}}(\D A)\\ \vec_{\mathcal{T}}(\D B)\\ \vec(\D D)}$ and consider
		\begin{align}\label{eq411}
			\bm{\zeta}\left([\mathfrak{D}_{\E}(\bm{w}^{\T}+\D \bm{w}^{\T}), f^{\T}+\D f^{\T},g^{\T}+\D g^{\T}]\right)-\bm{\zeta}\left([\mathfrak{D}_{\E}\bm{w}^{\T},\,  f^{\T},\, g^{\T}]^{\T}\right)= \mathbf{L}\bmatrix{\D x\\ \D y}.
		\end{align}
		Then from Lemma \ref{lm21}, we obtain
		\begin{align}\nonumber
			\bmatrix{\D x\\ \D y}&\approx -\M^{-1}\bmatrix{\mathcal{R} & -I_{m+n}}\bmatrix{\vec(\D A) \\ \vec(\D B)\\ \vec(\D D)\\
				\D f \\ \D g
			}\\ \nonumber
			&=-\M^{-1}\bmatrix{\mathcal{R} & -I_{m+n}}\bmatrix{\bm{\Phi}_{\SY_m} & {\bf 0} & {\bf 0} \\
			 {\bf 0}	& \bm{\Phi}_{\mathcal{T}_{nm}}& {\bf 0} \\  {\bf 0}&  {\bf 0}&I_{n^2+m+n}}\bmatrix{ \vec_{\SY}(\D A) \\ \vec_{\mathcal{T}}(\D B)\\  \vec(\D D) \\
				\D f \\ \D g
			}\\\nonumber
			&=-\M^{-1}\bmatrix{\mathcal{R}\bm{\Phi}_{\E} & -I_{m+n}}\bmatrix{\mathfrak{D}^{-1}_{\E}\mathfrak{D}_{\E}\D \bm{w}\\ \D f\\ \D g}\\ \label{eq318}
			&=-\M^{-1}\bmatrix{\mathcal{R}\bm{\Phi}_{\E} \mathfrak{D}^{-1}_{\E}& -I_{m+n}}\bmatrix{\mathfrak{D}_{\E}\D\bm{w}\\ \D f\\ \D g}.
		\end{align}
		Combining (\ref{eq318})  and (\ref{eq411}),  the $Fr\acute{e}chet$ derivative of $\bm{\zeta}$ at $\bmatrix{\mathfrak{D}_{\E}\bm{w}\\ f\\ g}$ is 
		\begin{align*}
			\d \bm{\zeta}\left([\mathfrak{D}_{\E}\bm{w}^{\T},\,  f^{\T},\, g^{\T}]^{\T}\right)=-\mathbf{L}\M^{-1}\bmatrix{\mathcal{R}\bm{\Phi}_{\E} \mathfrak{D}^{-1}_{\E}& -I_{m+n}}. 
		\end{align*} 
   {Hence, the proof follows.}
 \end{proof}

 \vspace{1.5mm}
	Using the Lemma \ref{lmm33} and Lemma \ref{lm21}, we next derive the compact formulae for the structured \textit{CNs} defined in Definition \ref{Def32}.
 \vone
	\begin{theorem}\label{Th32}
	Let $[x^{\T},\, y^{\T}]^{\T}$ be the unique solution of the \textit{GSPP} \eqref{Eq22} with the structure $\E$. Then, the structured \textit{NCN}, \textit{MCN}, and \textit{CCN} for the linear function $\mathbf{L}[x^{\T},\, y^{\T}]^{\T},$ respectively,  are given by
		\begin{eqnarray*}
&&\K(\mathbf{L}\bmatrix{x^{\T},\,  y^{\T}}^{\T};\,\E)=\frac{\Vr\mathbf{L}\M^{-1}\bmatrix{\mathcal{R}\bm{\Phi}_{\E} \mathfrak{D}^{-1}_{\E} & -I_{m+n}}\Vr_{2}\left\|\bmatrix{\mathbf{ H}& \bf{d}}\right\|_F}{\left\|\mathbf{L}\bmatrix{x^{\T},\,  y^{\T}}^{\T}\right\|_2},\\
	&&\Ms(\mathbf{L}\bmatrix{x^{\T},\,  y^{\T}}^{\T};\,\E)=\frac{\Bigg\| |\mathbf{L}\M^{-1}\mathcal{R}\bm{\Phi}_{\E}|\bmatrix{\vec_{\mathcal{S}}(|A|)\\ \vec_{\mathcal{T}}(|B|)\\ \vec(|D|)}+|\mathbf{L}\M^{-1}|\bmatrix{|f|\\|g|}\Bigg\|_{\infty}}{\Vr\mathbf{L}\bmatrix{x^{\T},\,  y^{\T}}^{\T}\Vr_{\infty}},\\
	&&\C(\mathbf{L}\bmatrix{x^{\T},\,  y^{\T}}^{\T};\,\E)=\left\| \mathbf{D}^{\dagger}_{\mathbf{L}[x^{\T},y^{\T}]^{\T}} |\mathbf{L}\M^{-1}\mathcal{R}\bm{\Phi}_{\E}|\bmatrix{\vec_{\mathcal{S}}(|A|)\\ \vec_{\mathcal{T}}(|B|)\\ \vec(|D|)}+ \mathbf{D}^{\dagger}_{\mathbf{L}[x^{\T},y^{\T}]^{\T}} |\mathbf{L}\M^{-1}|\bmatrix{|f|\\|g|}\right\|_{\infty}.
		\end{eqnarray*}
	\end{theorem}
	\begin{proof}
	 Let $\bm{w}^{\T}=[\vec(A)_{\mathcal{S}}^{\T}, \vec(B)_{\mathcal{T}}^{\T}, \vec(D)^{\T}]^{\T}.$ Following the proof method of Proposition \ref{PROP1}, we have 
  \vspace{-1mm}
		\begin{align*}
			&\K(\mathbf{L}\bmatrix{x^{\T},\,  y^{\T}}^{\T};\,\E)=\K(\bm{\zeta}, [\mathfrak{D}_{\E}\bm{w}^{\T},\,  f^{\T},\, g^{\T}]^{\T}), \quad \Ms(\mathbf{L}\bmatrix{x^{\T},\,  y^{\T}}^{\T};\,\E)=\Ms(\bm{\zeta}, [\mathfrak{D}_{\E}\bm{w}^{\T},\,  f^{\T},\, g^{\T}]^{\T})\\ 
			&\mbox{and}  \quad \C(\mathbf{L}\bmatrix{x^{\T},\,  y^{\T}}^{\T};\,\E)=\C(\bm{\zeta}, [\mathfrak{D}_{\E}\bm{w}^{\T},\,  f^{\T},\, g^{\T}]^{\T}).
		\end{align*}
  Applying the \textit{NCN} formula given in  Lemma \ref{lm21} for the map $\bm{\zeta},$ we obtain
  \begin{align}\label{TH:eq35}
      \K(\mathbf{L}[x^{\T},\, y^{\T}]^{\T};\,\E)&= \frac{\Vr \d \bm{\zeta}\left([\mathfrak{D}_{\E}\bm{w}^{\T},\,  f^{\T},\, g^{\T}]^{\T}\right) \Vr_{2}\left\|\bmatrix{\mathfrak{D}_{\E}\bm{w}\\f\\g}\right\|_2}{\left\|\bm{\zeta}\left([\mathfrak{D}_{\E}\bm{w}^{\T},\,  f^{\T},\, g^{\T}]^{\T}\right)\right\|_2}.
  \end{align}
  Now, substituting $Fr\acute{e}chet$ derivative of $\bm{\zeta}$  provided in Lemma \ref{lmm31} in \eqref{TH:eq35}, we have
		\begin{align*}
			\K(\mathbf{L}[x^{\T},\, y^{\T}]^{\T};\,\E)&=\frac{\Vr\mathbf{L}\M^{-1}\bmatrix{\mathcal{R}\bm{\Phi}_{\E} \mathfrak{D}^{-1}_{\E} & -I_{m+n}}\Vr_{2}\left\|\bmatrix{ \mathbf{H} & \bf{d}}\right\|_F}{\left\|\mathbf{L}[x^{\T},\, y^{\T}]^{\T}\right\|_2}.
		\end{align*}
Similarly, applying the \textit{MCN} formula provided in Lemma \ref{lm21} for $\bm{\zeta}$, we get 
  \begin{align}\label{TH:eq36}
      \Ms(\mathbf{L}[x^{\T},\, y^{\T}]^{\T};\,\E)&=\frac{\left\||\d\bm{\zeta}\left([\mathfrak{D}_{\E}\bm{w}^{\T},\,  f^{\T},\, g^{\T}]^{\T}\right) |\left| \bmatrix{\mathfrak{D}_{\E}\bm{w}\\f\\g}\right|\right\|_{\infty}}{\left\|\bm{\zeta}\left([\mathfrak{D}_{\E}\bm{w}^{\T},\,  f^{\T},\, g^{\T}]^{\T}\right)\right\|_{\infty}}.
  \end{align}
	Now, using Lemma \ref{lmm33} in \eqref{TH:eq36}, we obtain
		\begin{align*}
			\Ms(\mathbf{L}[x^{\T},\, y^{\T}]^{\T};\,\E)&=\frac{\left\|\vr\mathbf{L}\M^{-1}\bmatrix{\mathcal{R}\bm{\Phi}_{\E} \mathfrak{D}^{-1}_{\E} & -I_{m+n}}\vr\left |\bmatrix{\mathfrak{D}_{\E}\bm{w}\\f\\g}\right|\right\|_{\infty}}{\left\|\mathbf{L}[x^{\T},\, y^{\T}]^{\T}\right\|_{\infty}}\\
			&{ =\frac{\left\| |\mathbf{L} \M^{-1}\mathcal{R}\bm{\Phi}_{\E} \mathfrak{D}^{-1}_{\E} | |\mathfrak{D}_{\E}\bm{w}|+|\mathbf{L}\M^{-1}| \bmatrix{|f|\\ |g|}\right\|_{\infty}}{\left\|\mathbf{L} [x^{\T},\, y^{\T}]^{\T}\right\|_{\infty}}}\\
			&=\frac{\Bigg\| |\mathbf{L}\M^{-1}\mathcal{R}\bm{\Phi}_{\E}|\bmatrix{\vec_{\SY}(|A|)\\ \vec_{\mathcal{T}}(|B|)\\ \vec(|D|)}+|\mathbf{L}\M^{-1}|\bmatrix{|f|\\|g|}\Bigg\|_{\infty}}{\Vr\mathbf{L}[x^{\T},\, y^{\T}]^{\T}\Vr_{\infty}}.
		\end{align*}
		In an analogous method, we  get 
		\begin{align*}
			\C(\mathbf{L}[x^{\T},\, y^{\T}]^{\T};\,\E)&=\left\|\mathbf{D}^{\dagger}_{\mathbf{L}[x^{\T},y^{\T}]^{\T}} |\d\bm{\zeta}\left([\mathfrak{D}_{\E}\bm{w}^{\T},\,  f^{\T},\, g^{\T}]^{\T}\right) |\left| \bmatrix{\mathfrak{D}_{\E}\bm{w}\\f\\g}\right|\right\|_{\infty}\\
			&=\left\|\mathbf{D}^{\dagger}_{\mathbf{L}[x^{\T},y^{\T}]^{\T}} \mathbf{L}|\M^{-1}\mathcal{R}\bm{\Phi}_{\E}|\bmatrix{\vec_{\SY}(|A|)\\ \vec_{\mathcal{T}}(|B|)\\ \vec(|D|)}+ \mathbf{D}^{\dagger}_{\mathbf{L}[x^{\T},y^{\T}]^{\T}} |\mathbf{L}\M^{-1}|\bmatrix{|f|\\|g|}\right\|_{\infty}.
		\end{align*}
		Hence, the proof is completed.    
	\end{proof}
  \vone
{ \begin{remark}\label{re1:Th32}
    Note that the structured \textit{MCN} and \textit{CCN} formulae presented in Theorem \ref{Th32} involve computing the inverse of the matrix $\M \in \R^{(m+n)\times (m+n)}$, while the structured \textit{NCN} formula involves computing the inverse of both matrices $\M$ and $\mathfrak{D}_{\mathcal{E}}\in \R^{{\bm l}\times {\bm l}}$. 
     However, $\mathfrak{D}_{\mathcal{E}}$ is a diagonal matrix. Therefore, its inverse can be computed using only $\mathcal{O}({\bm l})$ operations. On the other hand, to avoid computing $\M^{-1}$ explicitly, motivated by \cite{Hanyu2018}, we adopt the following procedure. Notably, the computation of $\M^{-1}$ is coming in the following form: $\mathbf{L}\M^{-1}\bmatrix{\mathcal{R}\bm{\Phi}_{\E} \mathfrak{D}^{-1}_{\E} & -I_{m+n}}$ $\text{or}~ \mathbf{L}\M^{-1}\mathcal{R}\bm{\Phi}_{\E}~\text{or}~ \mathbf{L} \M^{-1}$. 
     Thus, first, we solve the system $\M X=Y,$ where $Y=\bmatrix{\mathcal{R}\bm{\Phi}_{\E} \mathfrak{D}^{-1}_{\E} & -I_{m+n}}$ $\text{or}~ \mathcal{R}\bm{\Phi}_{\E}$  and then compute $\mathbf{L}X.$ The system $\M X=Y$ can be solved efficiently by LU decomposition. To compute $ \mathbf{L} \M^{-1},$ we can solve $\mathbf{L}=XM.$ It is worth noting that we only need to perform the LU decomposition once for all cases; this makes the procedure efficient and reliable.
 \end{remark}}
 \vone
 \begin{remark}\label{re:Th32}
     The Toeplitz matrix $B$ is symmetric-Toeplitz (a special case of Toeplitz matrix) if $n=m$ and $b_{-n+1}=b_{n-1},\ldots, b_{-1}=b_1,$ where $\vec_{\mathcal{T}}(B)=[b_{-n+1},\ldots, b_1,b_0,\ldots, b_{n-1}]^{\T}.$ In this case, the basis  for the set of symmetric-Toeplitz matrices is defined as  $\left\{\widetilde{\mathcal{J}}_{i}\right\}_{i=1}^{n},$  where $\widetilde{\mathcal{J}}_{1}={\tt Toep}([(e_{n}^{(n)})^{\T},\,{\bf 0}]^{\T} )$ and 	$\widetilde{\mathcal{J}}_{(i+1)}={\tt Toep}([(e_{n-i}^{(n)})^{\T},\, (e^{(n-1)}_{i})^{\T}]^{\T}) \, \mbox{for}\, \, i=  1, \ldots, n-1.$  Hence,  the structured \textit{CNs} for the \textit{GSPP} \eqref{eq11} when $A$ is symmetric, $B$ is symmetric-Toeplitz is given by the formulae as in Theorem \ref{Th32}, with $\bm{\Phi}_{\mathcal{T}_{nm}}=\bmatrix{\vec(\widetilde{\mathcal{J}}_{1}), \ldots, \vec(\widetilde{\mathcal{J}}_{n})}\in \R^{n^2\times n}$ and $\mathfrak{D}_{\mathcal{T}_{nm}}\in \R^{n\times n}$   with  $\mathfrak{D}_{\mathcal{T}_{nm}}(j,j)=\bm{\hat{a}}_j,$ where  $\bm{\hat{a}}=[\sqrt{n},\sqrt{2(n-1)},\sqrt{2(n-2)},\ldots,\sqrt{2}]^{\T}\in \R^{n}.$ 
 \end{remark}
 	\vone
  
	Next, we compare the structured \textit{CNs} with the unstructured ones given in \eqref{EQ25}--\eqref{EQ27}.
 \vspace{1mm}
	\begin{theorem}\label{TH33}
		With the above notation, when $\mathbf{L}=I_{m+n},$ we have the following relations:
  \begin{eqnarray*}
      &\K([x^{\T},\, y^{\T}]^{\T};\, \E)\leq  \K^u([x^{\T},\, y^{\T}]^{\T}),\, \Ms([x^{\T},\,  y^{\T}]^{\T};\, \E)\leq  \Ms^u([x^{\T},\,  y^{\T}]^{\T}) \\
      &\mbox{and} \quad \C([x^{\T},\,  y^{\T}]^{\T};\, \E)\leq  \C^u([x^{\T},\,  y^{\T}]^{\T}).
  \end{eqnarray*}
		
	\end{theorem}
\begin{proof}
Since $\mathbf{L}=I_{m+n},$ for the \textit{NCN}, using the properties of the spectral norm, we obtain
		\begin{align*}
			\Vr\M^{-1}\bmatrix{\mathcal{R}\bm{\Phi}_{\E} \mathfrak{D}^{-1}_{\E} & -I_{m+n}}\Vr_{2}&\leq \Vr\M^{-1}\bmatrix{\mathcal{R} & -I_{m+n}}\Vr_{2}\left\| \bmatrix{\bm{\Phi}_{\E} \mathfrak{D}^{-1}_{\E} & {\bf 0} \\ {\bf 0} & I_{m+n}} \right\|_{2}\\
			&=\Vr\M^{-1}\bmatrix{\mathcal{R} & -I_{m+n}}\Vr_{2}.
		\end{align*}
		The last equality is obtained by using the fact that $\|\bm{\Phi}_{\E} \mathfrak{D}^{-1}_{\E}\|_2=1.$ Hence, the first claim is achieved.
  
  Since $ \bm{\Phi}_{\E} $ has at most one nonzero entry in each row, we obtain 
  \begin{align*}
      |\M^{-1}\mathcal{R}\bm{\Phi}_{\E}|\bmatrix{\vec_{\mathcal{S}}(|A|)\\ \vec_{\mathcal{T}}(|B|)\\ \vec(|D|)}&\leq |\M^{-1}\mathcal{R}| |\bm{\Phi}_{\E}|\bmatrix{\vec_{\mathcal{S}}(|A|)\\ \vec_{\mathcal{T}}(|B|)\\ \vec(|D|)}\\
      & = |\M^{-1}\mathcal{R}| \bmatrix{|\bm{\Phi}_{\SY_m}|\vec_{\mathcal{S}}(|A|)\\ |\bm{\Phi}_{\mathcal{T}_{nm}}|\vec_{\mathcal{T}}(|B|)\\ \vec(|D|)}\\
      &=  |\M^{-1}\mathcal{R}| \bmatrix{|\bm{\Phi}_{\SY_m}\vec_{\mathcal{S}}(A)|\\ |\bm{\Phi}_{\mathcal{T}_{nm}}\vec_{\mathcal{T}}(B)|\\ \vec(|D|)}\\
      &=|\M^{-1}\mathcal{R}| \bmatrix{\vec(|A|)\\ \vec(|B|)\\ \vec(|D|)}.
  \end{align*}
  Therefore, from Theorem \ref{Th31}, we obtain $$\Ms([x^{\T},\,  y^{\T}]^{\T};\, \E) \leq \frac{\left\| |\M^{-1}\mathcal{R}|\bmatrix{\vec(|A|)\\ \vec(|B|)\\ \vec(|D|)}+ |\M^{-1}|\bmatrix{|f|\\|g|} \right\|_{\infty}}{\|[x^{\T},\, y^{\T}]^{\T}\|_{\infty}}= \Ms^u([x^{\T},\,  y^{\T}]^{\T})$$
  and $$\C([x^{\T},\,  y^{\T}]^{\T};\, \E) \leq \Bigg\| \mathbf{D}^{\dagger}_{[x^{\T},\, y^{\T}]^{\T}}|\M^{-1}\mathcal{R}|\bmatrix{\vec(|A|)\\ \vec(|B|)\\ \vec(|D|)}+ \mathbf{D}^{\dagger}_{[x^{\T},\, y^{\T}]^{\T}} |\M^{-1}|\bmatrix{|f|\\|g|}\Bigg\|_{\infty}= \C^u([x^{\T},\,  y^{\T}]^{\T}).$$
Hence, the proof is completed.   
\end{proof}


\vone
\section{Structured \textit{CNs} when $A$ and $D$ have linear structures}\label{sec4}
	In this section, we consider $\mathcal{L}_1\subseteq\R^{m\times m}$ and $\mathcal{L}_2\subseteq\R^{n\times n}$  are two distinct linear subspaces containing different classes of structured matrices.
 Suppose that the  $\dim(\L_1)=p$ and $\dim(\L_2)=s$ and the corresponding bases are $\{E_i\}_{i=1}^p$ and $\{F_i\}_{i=1}^s,$ respectively. Let $A\in \L_1$ and $D\in \L_2.$ Then there are unique vectors $$\vec_{\L_1}(A)=[a_1, a_2,\ldots, a_p]^{\T}\in \R^p\quad \text{and} \quad \vec_{\L_2}(D)=[d_1,d_2,\ldots, d_s]^{\T}\in \R^s$$ such that 
  \begin{equation}\label{EQ42}
        A=\sum_{i=1}^{p}a_i E_i \quad \text{and} \quad   D=\sum_{i=1}^{s}d_i F_i.
  \end{equation}
  
  Subsequently, we obtain the following for the vec-structure of the matrices $A$ and $D.$
\vspace{1mm}
  \begin{lemma}
      Let $A\in \L_1$ and $D\in \L_2,$ then $\vec(A)=\bm{\Phi}_{\L_1}\vec_{\L_1}(A)$ and $\vec(D)=\bm{\Phi}_{\L_2}\vec_{\L_2}(D),$ where 
      \begin{eqnarray*}
\bm{\Phi}_{\L_1} &=& \bmatrix{\vec(E_1) & \vec(E_2) & \cdots &\vec(E_p)}\in \R^{m^2\times p}, \\
\bm{\Phi}_{\L_2} &=& \bmatrix{\vec(F_1) & \vec(F_2) & \cdots &\vec(F_s)}\in \R^{n^2\times s}.
\end{eqnarray*}
\end{lemma}
  \begin{proof}
  Assume that $\vec_{\L_1}(A)=[a_1, a_2,\ldots, a_p]^{\T}\in \R^p,$ then from \eqref{EQ42}, we obtain \begin{align*}
      \vec(A)&= \sum_{i=1}^{p}a_i\vec(E_i)= \bm{\Phi}_{\L_1}\vec_{\L_1}(A).
  \end{align*}
  Similarly, we can obtain $\vec(D)=\bm{\Phi}_{\L_2}\vec_{\L_2}(D).$      
  \end{proof}
\vone
  
  The matrices $\bm{\Phi}_{\L_1}$ and $\bm{\Phi}_{\L_2}$ contains the information about the structures of $A$ and $D$ consisting with the linear subspaces $\L_1$ and $\L_2,$ respectively. For unstructured matrices, $\bm{\Phi}_{\L_1}=I_{m^2}$ and $\bm{\Phi}_{\L_2}=I_{n^2}.$
	On the other hand, there exist  diagonal  matrices $\mathfrak{D}_{\L_1}\in \R^{p\times p}$  and $\mathfrak{D}_{\L_2}\in \R^{s\times s}$  {with the diagonal entries $\mathfrak{D}_{\L_j}(i,i)=\|\bm{\Phi}_{\L_j}(:,i)\|_2,$ for $j=1,2,$ such that }
	\begin{align}\label{eq45}
		\|A\|_F=\|\mathfrak{D}_{\L_1}a\|_2  \quad \text{and} \quad \|D\|_F=\|\mathfrak{D}_{\L_2}d\|_2.
	\end{align}

	To perform structured perturbation analysis, we restrict the perturbation $\D A$ on  $A$ and $\D D$ on $D$ to the linear subspaces $ \L_1$ and $\L_2,$ respectively. Then, there are unique vectors $ \vec_{\L_1}(\D A)\in \R^{p}$ and $ \vec_{\L_2}(\D D)\in \R^{s}$ such that
 \begin{equation}
     \vec(\D A)=\bm{\Phi}_{\L_1}\vec_{\L_1}(\D A) \, \, \text{and}\, \, \vec(\D D)=\bm{\Phi}_{\L_2}\vec_{\L_1}(\D D).
 \end{equation}

 Now, consider the following set: 
	\begin{equation}\label{eq47}
		{\bf \L}=\left\{\M=\bmatrix{A & B^{\T} \\C & D} :  A\in \L_1, B,C\in \R^{n\times m}, D\in \L_2\right\}.
	\end{equation} 
 Consider the perturbations $\D A, $ $\D B$, $\D C,$ $\D D,$ $\D f,$ and $\D g$ on the matrices $A,$ $B,$ $C,$ $D,$ $f,$ and $g,$ respectively. Then,  the perturbed counterpart of the system \eqref{eq11} 
	\begin{align}\label{eq22}
		(\M+\D \M)\bmatrix{x+\D x\\ y+\D y}=\bmatrix{A+\D A & (B+\D B)^{\T}\\ C+\D C & D+\D D}\bmatrix{x+\D x\\ y+\D y}=\bmatrix{f+\D f\\ g+\D g}
	\end{align}
	has a unique solution $\bmatrix{x+\D x \\ y+\D y}$  when $\|\M\|_2\left\| \D\M\right\|_2< 1.$
	Consequently, neglecting higher-order terms, we can rewrite (\ref{eq22}) as 
	
	\begin{align}\label{eq32}
		\M\bmatrix{\D x\\ \D y}=\bmatrix{A & B^{\T}\\ C& D}\bmatrix{\D x\\ \D y}=\bmatrix{\D f\\ \D g}-	
		\bmatrix{\D A & \D B^{\T}\\ \D C & \D D}\bmatrix{x \\ y}. 
	\end{align}
 
	Using the properties of the Kronecker product   { mentioned in} (\ref{eq13}), we have the following lemma.
  \vspace{1mm}
\begin{lemma}\label{lm22}
    Let $\bmatrix{x^{\mathsf{T}}, y^{\mathsf{T}}}^{\mathsf{T}}$ and $\bmatrix{(x+\D x)^{\mathsf{T}}, (y+\D y)^{\mathsf{T}}}^{\mathsf{T}}$ be the unique solutions of the \textit{GSPP} \eqref{eq11} and \eqref{eq22}, respectively, with structure $\L$. Then, we have the following perturbation expression
    \begin{align}
        \bmatrix{\D x\\ \D y}&\approx -\M^{-1}\bmatrix{\mathcal{H} & -I_{m+n}}\bmatrix{\vec(\D A) \\ \vec(\D B)\\ \vec(\D C)\\ \vec(\D D)\\ \D f\\ \D g
		},
    \end{align}
    where \begin{align}\label{eq36}
		\mathcal{H}=\bmatrix{x^{\T}\otimes I_{m} & I_m\otimes y^{\T} &{\bf 0} & {\bf 0}\\
			{\bf 0} & {\bf 0} & x^{\T}\otimes I_{n} & y^{\T}\otimes I_{n}}.
	\end{align}
\end{lemma}	

  Next, we define the structured \textit{NCN}, \textit{MCN}, and \textit{CCN}  for the linear function $\mathbf{L}[x^{\T},\, y^{\T}]^{\T}$ of the solution  of the \textit{GSPP} (\ref{eq11}) with the structure $\L.$ 
  \vspace{1mm}
 \begin{definition}\label{def41}
 Let $\bmatrix{x^{\T}, \, y^{\T}}^{\T}$ and  $\bmatrix{(x+\D x)^{\T}, \, (y+\D y)^{\T}}^{\T}$ be the unique solutions of \textit{GSPPs} \eqref{eq11} and \eqref{eq22}, respectively, with the  structure $\L$. Suppose $\mathbf{L}\in \R^{k\times (m+n)},$ then the structured  \textit{NCN}, \textit{MCN}, and \textit{CCN} for $\mathbf{L}[x^{\T},\, y^{\T}]^{\T},$ respectively, are defined as follows:
     \begin{align}\label{eq38}
		\nonumber &\K(\mathbf{L}[x^{\T},\, y^{\T}]^{\T};\,\L):=\lim_{\eta\rightarrow 0}\sup\left\{\frac{\left\|\mathbf{L}[\D x^{\T}, \, \D y^{\T}]^{\T}\right\|_2}{\eta \left\|\mathbf{L}[x^{\T},\, y^{\T}]^{\T}\right\|_2}:  \left\| \bmatrix{\D \M &\D \mathbf{d}}\right\|_F\leq \eta \left\|\bmatrix{ \M & \mathbf{d}}\right\|_F,\, \D \M\in \L\right\},\\
		&\nonumber	\Ms(\mathbf{L}[x^{\T},\, y^{\T}]^{\T};\,\L):=\lim_{\eta\rightarrow 0}\sup\left\{\frac{\left\|\mathbf{L}[\D x^{\T}, \, \D y^{\T}]^{\T}\right\|_{\infty}}{\eta \left\|\mathbf{L}[x^{\T},\, y^{\T}]^{\T}\right\|_{\infty}}: \, \left|\bmatrix{\D \M& \D \mathbf{d}}\right|\leq \eta \left|\bmatrix{\M&  \mathbf{d}}\right| ,\, \D \M \in \L\right\},\\
		&\nonumber	\C(\mathbf{L}[x^{\T},\, y^{\T}]^{\T};\,\L):=\lim_{\eta\rightarrow 0}\sup\left\{\frac{1}{\eta}\left\|\frac{\mathbf{L}[\D x^{\T},\, \D y^{\T}]^{\T}}{\mathbf{L}[x^{\T},  \, y^{\T}]^{\T}}\right\|_{\infty}: \, \left|\bmatrix{\D \M& \D \mathbf{d}}\right|\leq \eta \left|\bmatrix{\M&  \mathbf{d}}\right|,\, \D \M\in \L\right\}.
	\end{align}	
 \end{definition}

	

	The main objective of this section is to develop explicit formulae for the structured \textit{CNs} defined above. To accomplish these, let $\bm{v}$ be a vector in $\R^{p+2mn+s}$ defined as 
	\begin{equation}
	\bm{v}=\bmatrix{\vec_{\L_1}^{\T}(A), \vec(B)^{\T},  \vec(C)^{\T},  \vec_{\L_2}^{\T}(D)}^{\T}.
	\end{equation}
	To apply the Lemma \ref{lm21}, we define the mapping \begin{align}\label{eq48}
		& \Upsilon: \, \R^{p+2mn+s}\times \R^m\times \R^n \mapsto \R^{k} \quad \text{by}\\ \nonumber
		& \Up\left([\mathfrak{D}_{\L}\bm{v}^{\T},\,  f^{\T},\, g^{\T}]^{\T}\right)=\mathbf{L}\bmatrix{x \\ y}=\mathbf{L}\M^{-1}\bmatrix{f\\ g},
	\end{align}
	where \begin{align}\label{eq49}
		\mathfrak{D}_{\L}=\bmatrix{\mathfrak{D}_{\L_1} & {\bf 0}& {\bf 0}\\
			{\bf 0}& I_{2mn}&{\bf 0} \\
		{\bf 0}	& {\bf 0}& \mathfrak{D}_{\L_2}}
	\end{align} such that $\|\M\|_F=\|\mathfrak{D}_{\L}\bm{v}\|_2.$
 
	\vspace{1mm}
In the following lemma, we present explicit formulations of $\d\Up.$
\vone
	\begin{lemma}\label{lm31}
		The mapping $\Up$ defined in (\ref{eq48}) is continuously $Fr\acute{e}chet$ differentiable at $[\mathfrak{D}_{\L}\bm{v}^{\T},\,  f^{\T},\, g^{\T}]^{\T}$  and the $Fr\acute{e}chet$ derivative is given by 
		\begin{align}
\d\Up\left([\mathfrak{D}_{\L}\bm{v}^{\T},\,  f^{\T},\, g^{\T}]^{\T}\right)=-\mathbf{L}\M^{-1}\bmatrix{\mathcal{H}\bm{\Phi}_{\mathbf{\L}} \mathfrak{D}_{\L} & -I_{m+n}},
		\end{align}
		where ${ \bm{\Phi}}_{\L}=\bmatrix{\bm{\Phi}_{\L_1} & {\bf 0}&{\bf 0}\\
			{\bf 0}& I_{2mn} &{\bf 0} \\
			{\bf 0}&{\bf 0}& \bm{\Phi}_{\L_2}},$ $\mathcal{H}$  and $\mathfrak{D}_{\L}$ are defined as in $(\ref{eq36})$ and $(\ref{eq49}),$ respectively.  
	\end{lemma}
 \begin{proof}
     The proof follows in a similar way to the proof method of Lemma \ref{lmm33}.
 \end{proof}

	We now present compact formulae of the structured \textit{NCN}, \textit{MCN}, and \textit{CCN} introduced in Definition \ref{def21}. We use the Lemmas \ref{lm21} and $\ref{lm31}$ to prove the following theorem.
	\begin{theorem}\label{th31}
		The structured \textit{NCN}, \textit{MCN}, and \textit{CCN} for the linear function $\mathbf{L}[x^{\T},\, y^{\T}]^{\T}$ of the solution of the \textit{GSPP} $(\ref{eq11})$ with the  structure $\L,$ respectively,  are given by
		\begin{eqnarray*}	
 && \K(\mathbf{L}[x^{\T},\, y^{\T}]^{\T};\,\L)=\frac{\Vr\mathbf{L}\M^{-1}\bmatrix{\mathcal{H}\bm{\Phi}_{\L} \mathfrak{D}^{-1}_{\L} & -I_{m+n}}\Vr_{2}\left\|\bmatrix{ \M & \mathbf{d}}\right\|_F}{\left\|\mathbf{L}[x^{\T},\, y^{\T}]^{\T}\right\|_2},\\
		&&	\Ms(\mathbf{L}[x^{\T},\, y^{\T}]^{\T};\,\L)=\frac{\Bigg\| |\mathbf{L}\M^{-1}\mathcal{H}\bm{\Phi}_{\L}|\bmatrix{\vec_{\L_1}(|A|)\\ \vec(|B|)\\ \vec(|C|)\\ \vec_{\L_2}(|D|)}+|\mathbf{L}\M^{-1}|\bmatrix{|f|\\|g|}\Bigg\|_{\infty}}{\Vr\mathbf{L}[x^{\T},\, y^{\T}]^{\T}\Vr_{\infty}},\\
		&&\C(\mathbf{L}[x^{\T},\, y^{\T}]^{\T};\,\L)=\Bigg\|\mathbf{D}^{\dagger}_{\mathbf{L}[x^{\T},y^{\T}]^{\T}} |\mathbf{L}\M^{-1}\mathcal{H}\bm{\Phi}_{\L}|\bmatrix{\vec_{\L_1}(|A|)\\ \vec(|B|)\\ \vec(|C|)\\ \vec_{\L_2}(|D|)}+\mathbf{D}^{\dagger}_{\mathbf{L}[x^{\T},y^{\T}]^{\T}} |\mathbf{L}\M^{-1}|\bmatrix{|f|\\|g|}\Bigg\|_{\infty}.
		\end{eqnarray*}
	\end{theorem}
	\begin{proof}
	Similar { to} the proof method of Proposition \ref{PROP1}, we have\\
			$\K(\mathbf{L}[x^{\T},\, y^{\T}]^{\T};\, \L)=\K(\Up,\bmatrix{\mathfrak{D}_{\L}\bm{v}^{\T}, f^{\T},g^{\T}}^{\T}),$ $ \Ms(\mathbf{L}[x^{\T},$ $ y^{\T}]^{\T};\, \L)=\Ms(\Up,\bmatrix{\mathfrak{D}_{\L}\bm{v}^{\T}, f^{\T},g^{\T}}^{\T}),$ and
				$\C(\mathbf{L}[x^{\T},\, y^{\T}]^{\T};\,\L)=\C(\Up,\bmatrix{\mathfrak{D}_{\L}\bm{v}^{\T}, f^{\T},g^{\T}}^{\T}).$
    
   \noindent Using Lemma \ref{lm31} and  {\textit{NCN}} formula provided in Lemma \ref{lm21}, we have
		\begin{align*}
			\K(\mathbf{L}[x^{\T},\, y^{\T}]^{\T};\L)&= \frac{\Vr \d\Up\left([\mathfrak{D}_{\L}\bm{v}^{\T},\,  f^{\T},\, g^{\T}]^{\T}\right) \Vr_{2}\left\|\bmatrix{\mathfrak{D}_{\L}\bm{v}\\f\\g}\right\|_2}{\left\|\Up\left([\mathfrak{D}_{\L}\bm{v}^{\T},\,  f^{\T},\, g^{\T}]^{\T}\right)\right\|_2}\\
			&=\frac{\Vr\mathbf{L}\M^{-1}\bmatrix{\mathcal{H}\bm{\Phi}_{\L} \mathfrak{D}^{-1}_{\L} & -I_{m+n}}\Vr_{2}\left\|\bmatrix{ \M & \mathbf{d}}\right\|_F}{\left\|\mathbf{L}[x^{\T},\, y^{\T}]^{\T}\right\|_2}.
		\end{align*}
		For structured \textit{MCN}, again using Lemmas \ref{lm21} and \ref{lm31}, we obtain
		
		\begin{align*}
			\Ms(\mathbf{L}[x^{\T},\, y^{\T}]^{\T};\L)&=\frac{\left\||\d\Up\left([\mathfrak{D}_{\L}\bm{v}^{\T},\,  f^{\T},\, g^{\T}]^{\T}\right) |\left| \bmatrix{\mathfrak{D}_{\L}\bm{v}\\f\\g}\right|\right\|_{\infty}}{\left\|\Up\left([\mathfrak{D}_{\L}\bm{v}^{\T},\,  f^{\T},\, g^{\T}]^{\T}\right)\right\|_{\infty}}\\
			&=\frac{\left\|\vr\mathbf{L}\M^{-1}\bmatrix{\mathcal{H}\bm{\Phi}_{\L} \mathfrak{D}^{-1}_{\L} & -I_{m+n}}\vr\left |\bmatrix{\mathfrak{D}_{\L}\bm{v}\\f\\g}\right|\right\|_{\infty}}{\left\|\mathbf{L}[x^{\T},\, y^{\T}]^{\T}\right\|_{\infty}}\\
    &=\frac{\Bigg\| |\mathbf{L}\M^{-1}\mathcal{H}\bm{\Phi}_{\L}|\bmatrix{\vec_{\L_1}(|A|)\\ \vec(|B|)\\ \vec(|C|)\\ \vec_{\L_2}(|D|)}+|\mathbf{L}\M^{-1}|\bmatrix{|f|\\|g|}\Bigg\|_{\infty}}{\Vr\mathbf{L}[x^{\T},\, y^{\T}]^{\T}\Vr_{\infty}}.
  \end{align*}
  The rest of the proof follows similarly.
  \end{proof}
	\vspace{1mm}
 { \begin{remark}
     To compute the inverses of $\M$ and $\mathfrak{D}_{\L}$, one can follow a similar procedure as discussed in Remark \ref{re1:Th32}.
 \end{remark}}
 \begin{remark}
     Considering $\mathbf{L}= I_{m+n}, \bmatrix{I_m & \bf 0}$ and $\bmatrix{\bf 0& I_n}$ in Theorem \ref{th31}, we obtain the structured \textit{NCN}, \textit{MCN}, and \textit{CCN} for the solution $[x^{\T}, \, y^{\T}]^{\T}, x$ and $y,$ respectively.
 \end{remark}
	
\begin{remark}\label{RE:TH45}
For $A\in \SY_{m}$ and $D\in \SY_{n},$ set $$\bm{\Phi}_{\SY}=\bmatrix{\bm{\Phi}_{\SY_{m}} &{\bf 0} &{\bf 0}\\ {\bf 0}& I_{2mn} &{\bf 0} \\ {\bf 0}&{\bf 0} & \bm{\Phi}_{\SY_{n}}} \quad \text{and} \quad
\mathfrak{D}_{\SY}=\bmatrix{\mathfrak{D}_{\SY_{m}} & {\bf 0}&{\bf 0}\\{\bf 0} & I_{2mn} & {\bf 0}\\{\bf 0} & {\bf 0}& \mathfrak{D}_{\SY_{n}}},$$ where $\bm{\Phi}_{\SY_m}, \bm{\Phi}_{\SY_n}, \mathfrak{D}_{\SY_{m}}$ and $\mathfrak{D}_{\SY_{n}}$ are defined as in Subsection \ref{subsec32}.
 Then, the structured \textit{NCN}, \textit{MCN}, and \textit{CCN} when $\L_1=\SY_m$ and $\L_2=\SY_n$ are obtained by substituting   $\bm{\Phi}_{\L} =\bm{\Phi}_{\SY},$   $\mathfrak{D}_{\L} =\mathfrak{D}_{\SY},$ $\vec_{\L_1}(A)=\vec_{\SY_m}(A)$ and $\vec_{\L_2}(D)=\vec_{\SY_n}(D)$ in Theorem \ref{th31}. 
\end{remark}

 Next,  consider the linear system $\M z={\bf d},$ where $\M\in \R^{l\times l}$ being any nonsingular matrix and   ${\bf d}\in \R^{l}.$  Then, this system can be partitioned as \textit{GSPP} \eqref{eq11} by setting $l=m+n$.
      Let $\D \M\in \R^{l\times l}$ and $\D {\bf d}\in \R^l,$ then the perturbed system is given by
      \begin{equation*}
         (\M +\D \M)(z+\D z)=({\bf d}+\D {\bf d}). 
      \end{equation*} 
      
      \citet{skeel1979scaling} and \citet{rohn1989new} propose the following formulae for the unstructured \textit{MCN} and \textit{CCN} for the solution of the above linear system:
	\begin{eqnarray}
	\nonumber	\widetilde{\Ms}(z)&:=& \lim_{\eta\rightarrow 0}\sup\left\{{ \frac{\|\D z\|_{\infty}}{\eta\|z\|_{\infty}}}\, :    |\D \M|\leq \eta |\M|,  |\D {\bf d}|\leq \eta |{\bf d}|\right \}\\ \label{eq331}
  &=&\frac{\left\||\M^{-1}| |\M||z|+|\M^{-1}||{\bf d}|\right\|_{\infty}}{\|z\|_{\infty}},\\ \nonumber
			\widetilde{\C}(z)&:=& \lim_{\eta\rightarrow 0}\sup\left\{{ \frac{1}{\eta}\left\|\frac{\D z}{z}\right\|_{\infty} }:   |\D \M|\leq \eta |\M|,  |\D {\bf d}|\leq \eta |{\bf d}|\right\}\\ \label{eq332}
   &=&\left\|\frac{|\M^{-1}| |\M|z|+|\M^{-1}||{\bf d}|}{|z|}\right\|_{\infty}.
	\end{eqnarray}
 \begin{remark}\label{rm41}
     Considering $\bm{\Phi}_{\mathcal{S}_m}=I_{m^2}$ and $\bm{\Phi}_{\mathcal{S}_n}=I_{n^2}$ on the formula of $\K([x^{\T},\, y^{\T}]^{\T};\L),$ then we obtain the unstructured \textit{NCN} for $\M z= \bf{d},$  where  $\M \in \R^{l\times l},$ ${\bf{d}}\in \R^{l}$ and $l=(m+n),$ which is given by 
     \begin{align*}
        \widetilde{\K}(z)&:=\lim_{\eta\rightarrow 0}\sup\Bigg\{{ \frac{\|\D z\|_{2}}{\eta\|z\|_{2}}}\, \dm{:}    \left\| \bmatrix{\D \M &\D \bf{d}}\right\|_F\leq \eta \left\|\bmatrix{ \M & \mathbf{d}}\right\|_F\Bigg\}\\
         &=\frac{\Vr\M^{-1}\bmatrix{\mathcal{H} & -I_{m+n}}\Vr_{2}\left\|\bmatrix{ \M & \mathbf{d}}\right\|_F}{\| z\|_2}.
     \end{align*}   
 \end{remark}

	The following theorem compares the structured \textit{NCN}, \textit{MCN}, and \textit{CCN} obtained in Theorem \ref{th31} and the unstructured counterparts defined above.
	
 \vspace{1.5mm}
 \begin{theorem}\label{th42}
	Let $z=[x^{\T},\, y^{\T}]^{\T}$	and $\mathbf{L}=I_{m+n}.$  Then, for the \textit{GSPP} $(\ref{eq11})$ with the  structure $\L,$  following relations holds:
		\begin{align*}
	&\K([x^{\T},\, y^{\T}]^{\T};\, \L)\leq \widetilde{\K}([x^{\T},\, y^{\T}]^{\T}),	\quad	\Ms([x^{\T},\, y^{\T}]^{\T}; \, \L)\leq \widetilde{\Ms}([x^{\T},\, y^{\T}]^{\T}) \\
 &\quad \mbox{and} 	\quad \C([x^{\T},\, y^{\T}]^{\T}; \,\L)\leq  \widetilde{\C}([x^{\T},\, y^{\T}]^{\T}).
		\end{align*}
	\end{theorem}
 
	\begin{proof}
	Since  $\|\bm{\Phi}_{\L} \mathfrak{D}^{-1}_{\L}\|_2= 1,$ the proof follows similar to the proof method of Theorem \ref{TH33}.
  Hence, from  Theorem \ref{th31} and Remark \ref{rm41}, we have $\K([x^{\T},\, y^{\T}]^{\T}; \, \L)\leq \widetilde{\K}([x^{\T},\, y^{\T}]^{\T}).$
  Now, using the property that the matrices $\bm{\Phi}_{\L_i}, \, i= 1, 2,$ have at most one nonzero entry in each row \cite{STLS2011}, and similar to Theorem \ref{TH33}, we obtain
		$|\bm{\Phi}_{\L_1} \vec_{\L_1}(A)|= |\bm{\Phi}_{\L_1}| \vec_{\L_1}(|A|)$ and $|\bm{\Phi}_{\L_2} \vec_{\L_2}(D)|= |\bm{\Phi}_{\L_2}| \vec_{\L_2}(|D| ).$ Then
		\begin{align*}
			&|\M^{-1}\mathcal{H}\bm{\Phi}_{\L}|\bmatrix{\vec_{\L_1}(|A|)\\ \vec(|B|)\\ \vec(|C|)\\ \vec_{\L_2}(|D|)}+ |\M^{-1}|\bmatrix{|f|\\|g|}\leq  |\M^{-1}||\mathcal{H}|\bmatrix{\vec(|A|)\\ \vec(|B|)\\ \vec(|C|)\\ \vec(|D|) }+|\M^{-1}|\bmatrix{|f|\\|g|}\\
   &\leq |\M^{-1}|\bmatrix{|x^{\T}|\otimes I_{m} & I_m\otimes |y^{\T}| & {\bf 0} & {\bf 0}\\
				{\bf 0} & {\bf 0} & |x^{\T}|\otimes I_{n} & |y^{\T}|\otimes I_{n}}\bmatrix{\vec(|A|)\\ \vec(|B|)\\ \vec(|C|)\\ \vec(|D|)}+|\M^{-1}|\bmatrix{|f|\\|g|}\\
			&=|\M^{-1}||\M|\bmatrix{|x|\\ |y|}+ |\M^{-1}|\bmatrix{|f|\\|g|}.\label{eq442}
		\end{align*} 
		 {Consequently,} by Theorem \ref{th31}, we have
		\begin{align}
			&\Ms([x^{\T},\, y^{\T}]^{\T}; \,\L)\leq \frac{\left\| |\M^{-1}||\M|\bmatrix{|x|\\ |y|}+|\M^{-1}|\bmatrix{|f|\\|g|}\right\|_{\infty}}{\|[x^{\T},\, y^{\T}]^{\T}\|_{\infty}},\\
			&\C([x^{\T},\, y^{\T}]^{\T}; \, \L)\leq \left\| \frac{|\M^{-1}||\M|\bmatrix{|x|\\ |y|}+|\M^{-1}|\bmatrix{|f|\\|g|}}{|\bmatrix{x^{\T},\,  y^{\T}}^{\T}|}\right\|_{\infty}.
		\end{align}
		Now, considering $l= m+ n,$ $z= [x^{\T},\, y^{\T}]^{\T}$ and ${\bf{d}}= [f^{\T},\, g^{\T}]^{\T}$ in (\ref{eq331}) and (\ref{eq332}) {, and} from above, we obtain the following relations:
		\begin{align*}
			\Ms([x^{\T},\, y^{\T}]^{\T}; \, \L)\leq \widetilde{\Ms}([x^{\T},\, y^{\T}]^{\T}) \quad \mbox{and} 	\quad \C([x^{\T},\, y^{\T}]^{\T}; \,\L)\leq \widetilde{\C}([x^{\T},\, y^{\T}]^{\T}).
		\end{align*}
  { Hence, the proof follows.}
  \end{proof}
\vone
\section{Application to \textit{WTRLS}  problems}\label{WTRLS_problem}
Consider the \textit{WTRLS}  problem \eqref{eq:WTRLS} and 
let  ${\bf r}=W(f-Q{ y}),$ then the minimization problem \eqref{eq:WTRLS} can be expressed as the following augmented linear system 
\begin{align}\label{eq72}
	\widehat{\M}\bmatrix{{\bf r} \\ { y}}:= \bmatrix{W^{-1} & Q\\ Q^{\T} & -\lambda I_n}\bmatrix{{\bf r} \\ { y}}=\bmatrix{f\\ {\bf 0}}.
\end{align}
 Identifying $A=W^{-1},$ $B=Q^{\T},$ $D=-\lambda I_n,$ $x={\bf r},$    and $g={\bf 0},$ we can see that the augmented system (\ref{eq72}) is in the form of the \textit{GSPP} (\ref{Eq22}). Therefore, finding the \textit{CNs} of the $WTRLS $ problem \eqref{eq:WTRLS} is equivalent to the \textit{CNs} of the \textit{GSPP} \eqref{Eq22} for $y$ with $g={\bf 0}.$  This accomplish by Theorem \ref{Th31}. Before that,   we reformulate \eqref{eq12} (with $B=C$) as
 \begin{align}\label{EQ442}
		\M^{-1}=\bmatrix{M & N \\ K & S^{-1}},
	\end{align}
	where $M=A^{-1}+A^{-1}B^{\T}S^{-1}BA^{-1},$ $N=-A^{-1}B^{\T}S^{-1},$ $K=-S^{-1}BA^{-1}$ and $S=D-BA^{-1}B^{\T}.$
 
 \vspace{1mm}
 \begin{theorem}
  Let ${ y}$ be the unique solution of the problem \eqref{eq:WTRLS} and ${\bf r}=W(f-Q{ y}).$  Then, the structured \textit{NCN}, \textit{MCN}, and \textit{CCN} for  { ${ y}$, respectively,} are given by
  \begin{align*}
&\K^{rls}({ y};\E)=\frac{\left\| \bmatrix{({\bf r}^{\T}\otimes \widetilde{K})\bm{\Phi}_{\SY_m} \mathfrak{D}^{-1}_{\SY_m} & (K\otimes { y}^{\T}+{\bf r}^{\T}\otimes \widetilde{S}^{-1})\bm{\Phi}_{\mathcal{T}_{nm}} \mathfrak{D}^{-1}_{\E} & { y}^{\T}\otimes \widetilde{S}^{-1} & -\widetilde{K} & -\widetilde{S}^{-1}}\right\|_{2}}{\left\|{ y}\right\|_2 / \left\|\bmatrix{\widehat{\M}& \mathbf{d}}\right\|_F},\\
	&\hspace{3cm}\Ms^{rls}({ y}; \, \E)= \frac{ \|\mathcal{N}_{ y}\|_{\infty}}{\|{ y}\|_{\infty}},\quad
			\C^{rls}({ y};\E)= \left\| \mathbf{D}^{\dagger}_{{ y}} \mathcal{N}_{ y}\right\|_{\infty},
		\end{align*}
  where $\mathcal{N}_{ y}=|({\bf r}^{\T}\otimes \widetilde{K})\bm{\Phi}_{\mathcal{S}_{m}}|\vec_{\mathcal{S}}(|A|)+\, |((K\otimes { y}^{\T})+ \,({\bf r}^{\T}\otimes \widetilde{S}^{-1}))\bm{\Phi}_{\mathcal{T}_{nm}}| \vec_{\mathcal{T}}(|Q^{\T}|)+ |\widetilde{S}^{-1}||D| |{ y}|+ |K||f|,$ $\widetilde{K}=-\widetilde{S}^{-1}Q^{\T}W,$ and  $\widetilde{S}= -(\lambda I_n+Q^{\T}WQ).$
 \end{theorem}
 
 \vspace{1mm}
 \begin{proof}
 Let $\mathbf{L}= \bmatrix{{\bf 0}&I_n}\in \R^{n\times (m+n)},$ $A=W^{-1}, B=Q^{\T}, D=-\lambda I_n,$ $x={\bf r},$ and $g={\bf 0.}$ Then from Theorem \ref{Th31}, we have 
\begin{align*}
&\mathbf{L}\M^{-1}\bmatrix{\mathcal{R}\bm{\Phi}_{\E}\mathfrak{D}_{\E}^{-1} & -I_{m
            +n}}\\
            &=\bmatrix{\widetilde{K} & \widetilde{S}^{-1}}\bmatrix{\mathcal{R} &-I_{m+n}}\bmatrix{\bm{\Phi}_{\E}\mathfrak{D}^{-1}_{\E}& \mathbf{0}\\\mathbf{0} & I_{m+n}}\\ 
            &=\bmatrix{({\bf r}^{\T}\otimes K)\bm{\Phi}_{\SY_m} \mathfrak{D}^{-1}_{\SY_m} & (\widetilde{K}\otimes { y}^{\T}+{\bf r}^{\T}\otimes \widetilde{S}^{-1})\bm{\Phi}_{\mathcal{T}_{nm}} \mathfrak{D}^{-1}_{\mathcal{T}_{nm}} & { y}^{\T}\otimes \widetilde{S}^{-1} & -\widetilde{K} & -\widetilde{S}^{-1}}.
  \end{align*}
  Hence, the expression for $\K^{rls}({ y}; \,\E)$ is obtained from Theorem \ref{Th31}. The rest of the proof follows in a similar manner.  
  \end{proof}

 \vspace{1mm}
 Since, in most cases of the \textit{WTRLS} problem, the weighted matrix $W$ and regularization matrix $D= -\lambda I_n$ has no perturbation, we consider $\D A={\bf 0}$ and $\D D= {\bf 0}.$  Moreover, as $g={\bf 0},$ we assume $\D g= {\bf 0}.$ Then, perturbation expansion  in Lemma \ref{LM22} can be reformulated as 
\begin{align}
	\nonumber	\bmatrix{\D x\\ \D y}&= -\M^{-1}\bmatrix{I_m \otimes y^{\T} & -I_m\\ x^{\T}\otimes I_n & {\bf 0}} \bmatrix{\vec(\D B)\\ \D f}\\
 &= -\bmatrix{\mathcal{R}_{rls}& -\bmatrix{M\\ K}}\bmatrix{\vec(\D B)\\ \D f},
\end{align}
where $\mathcal{R}_{rls}=\bmatrix{M\otimes y^{\T}+x^{\T}\otimes N\\ K\otimes y^{\T}+x^{\T}\otimes S^{-1} }.$
Now, applying a similar method to Subsection \ref{subsec32}, we obtain the following expressions for the \textit{NCN}, \textit{MCN}, and \textit{CCN} for $\mathbf{L}[x^{\T},\, y^{\T}]^{\T}$ when $B=C$ and $g={\bf 0}.$

 \vspace{1mm}
\begin{theorem}\label{prop51}
	Let $\D B\in \mathcal{T}^{n\times m}$ and  with the above  {notations},  {structured} \textit{NCN}, \textit{MCN}, and \textit{CCN} for the \textit{GSPP} \eqref{Eq22}, respectively, are given by 
	   \begin{eqnarray*}
          \widehat{\K}(\mathbf{L}[x^{\T},\, y^{\T}]^{\T})&:=& \lim_{\eta\rightarrow 0}\sup\Bigg\{\frac{\|{\mathbf{L}}[\Delta x^{\T},\, \Delta y^{\T}]^{\T} \|_2}{\eta\|\mathbf{L}[x^{\T},\, y^{\T}]^{\T}\|_2}: \left\|\bmatrix{\Delta B &\D f}\right\|_F\leq \eta \left\|\bmatrix{B& f}\right\|_F \Bigg\}\\ &=&\frac{\left\|\mathbf{L}\bmatrix{\mathcal{R}_{rls}\bm{\Phi}_{\mathcal{T}_{nm}}\mathfrak{D}_{\mathcal{T}_{nm}}^{-1} & -\bmatrix{M\\ K}}\right\|_2 \left\|\bmatrix{B& f}\right\|_F}{\|\mathbf{L}[x^{\T},\, y^{\T}]^{\T}\|_2},\\
         \widehat{\Ms}(\mathbf{L}[x^{\T},\, y^{\T}]^{\T})&:=& \lim_{\eta\rightarrow 0}\sup\Bigg\{\frac{\|{\mathbf{L}}[\Delta x^{\T},\, \Delta y^{\T}]^{\T} \|_{\infty}}{\eta\|\mathbf{L}[x^{\T},\, y^{\T}]^{\T}\|_{\infty}}: \left\vert\bmatrix{\Delta B &\D f}\right\vert \, \leq \eta \left\vert \bmatrix{ B & f}\right\vert \Bigg\}\\
&=& \frac{\left\||\mathbf{L}\mathcal{R}_{rls}\bm{\Phi}_{\mathcal{T}_{nm}}| \vec_{\mathcal{T}}(|B|)+\left|\mathbf{L}\bmatrix{M\\ K}\right||f|\right\|_{\infty}}{\|\mathbf{L}[x^{\T},\, y^{\T}]^{\T}\|_{\infty}},\\
         \widehat{\C}(\mathbf{L}[x^{\T},\, y^{\T}]^{\T})&:=&\lim_{\eta\rightarrow 0}\sup\Bigg\{\frac{1}{\eta}\left\|\frac{\mathbf{L}[\Delta x^{\T},\, \Delta y^{\T}]^{\T}}{\mathbf{L}[x^{\T},\, y^{\T}]^{\T}}\right\|_{\infty}:  \left\vert\bmatrix{\Delta B &\D f}\right\vert \, \leq \eta \left\vert \bmatrix{B & f}\right\vert\Bigg\}\\
&=&\left\|\mathbf{D}^{\dagger}_{\mathbf{L}[x^{\T},\, y^{\T}]^{\T}}|\mathbf{L}\mathcal{R}_{rls}\bm{\Phi}_{\mathcal{T}_{nm}}| \vec_{\mathcal{T}}(|B|) + \mathbf{D}^{\dagger}_{\mathbf{L}[x^{\T},\, y^{\T}]^{\T}} \left|\mathbf{L}\bmatrix{M\\ K}\right||f|\right\|_{\infty}.
     \end{eqnarray*}
\end{theorem}
\begin{proof}
  For applying  Lemma \ref{lm21}, we define 
	\begin{align*}
		&\bm{\widehat{\zeta}}:\, \R^{m+n-1}\times \R^m \mapsto \R^{m+n} \quad \mbox{by} \\ &\bm{\widehat{\zeta}}\left([\mathfrak{D}_{\mathcal{T}_{nm}}\vec_{\mathcal{T}}(B)^{\T},\,  f^{\T}]^{\T}\right)=\textbf{L}\bmatrix{x\\ y}=\textbf{L}\M^{-1}\bmatrix{f\\ {\bf 0}}.
	\end{align*}
	  Then, the map $\bm{\widehat{\zeta}}$ is continuously $Fr\acute{e}chet$ differentiable at  $[\mathfrak{D}_{\mathcal{T}_{nm}}\vec_{\mathcal{T}}(B)^{\T},\,  f^{\T}]^{\T}$ with $$
\d\bm{\widehat{\zeta}}\left([\mathfrak{D}_{\mathcal{T}_{nm}}\vec_{\mathcal{T}}(B)^{\T},\,  f^{\T}]^{\T}\right)=-\mathbf{L}\bmatrix{\mathcal{R}_{rls}\bm{\Phi}_{\mathcal{T}_{nm}}\mathfrak{D}_{\mathcal{T}_{nm}}^{-1}& -\bmatrix{M\\ K}}.$$ The rest of the proof follows similarly to Theorem \ref{Th32}.
\end{proof}


Using the above result, we can derive the following structured \textit{CNs} for the problem \eqref{eq:WTRLS}, when the weighted matrix and regularization matrix have no perturbation.

 \vspace{1.5mm}
\begin{corollary}
	 The structured \textit{NCN}, MCM, and \textit{CCN}  {for the  solution ${ y}$ of the \textit{WTRLS} problem \eqref{eq:WTRLS}}, respectively, are given by  
  \begin{eqnarray*}
          \widehat{\K}^{rls}( { y}) &=& \frac{\left\|\bmatrix{(\widetilde{K}\otimes { y}^{\T}+ {\bf r}^{\T}\otimes \widetilde{S}^{-1})\bm{\Phi}_{\mathcal{T}_{nm}}\mathfrak{D}_{\mathcal{T}_{nm}}^{-1} & -\widetilde{K}}\right\|_2 \left\|\bmatrix{Q& f}\right\|_F}{\|{ y}\|_2},\\
          \widehat{\Ms}^{rls}({ y}) &=& \frac{\left\||(\widetilde{K}\otimes { y}^{\T}+ {\bf r}^{\T}\otimes \widetilde{S}^{-1}) \bm{\Phi}_{\mathcal{T}_{nm}}| \vec_{\mathcal{T}}(|Q^{\T}|)+|\widetilde{K}||f|\right\|_{\infty}}{\|{ y}\|_{\infty}},\\
         \widehat{\C}^{rls}({ y}) &=& \|\mathbf{D}^{\dagger}_{{ y}}|(\widetilde{K}\otimes { {y}}^{\T}+ {\bf r}^{\T}\otimes \widetilde{S}^{-1}) \bm{\Phi}_{\mathcal{T}_{nm}}| \vec_{\mathcal{T}}(|Q^{\T}|)  + \mathbf{D}^{\dagger}_{ {y}} |\widetilde{K}||f|\|_{\infty},
     \end{eqnarray*}
      where $\widetilde{K}=\widetilde{S}^{-1}Q^{\T}W$ and $\widetilde{S}=-(\lambda I_n+Q^{\T}WQ).$
\end{corollary}
\begin{proof}
     Substituting $\mathbf{L}=\bmatrix{{\bf 0}& I_n}\in \R^{n\times (m+n)},$ $B=Q^{\T}, A=W^{-1},\, D=-\lambda I_n$ and $x=W(f-Q{ y})$ in Theorem \ref{prop51},  the proof follows. 
     \end{proof}
     
 \vspace{1.5mm}
\begin{remark}
  We consider the Tikhonov regularization problem ${ \displaystyle{\min_{w\in \R^n}\left\{\|B^{\T}w-f\|^2_2+\lambda \|Rw\|_2^2\right\}}},$ where $R$ is the regularization matrix and  $\lambda>0$ regularization parameter. Then, substituting $\mathbf{L}=\bmatrix{{\bf 0}& I_n}\in \R^{n\times (m+n)},$ $ A=I_m, D=-\lambda R^{\T}R,$ $x=(f-B^{\T}w)$ and $y=w,$ in Theorem \ref{prop51}, we can recover the structured \textit{NCN}, \textit{MCN}, and \textit{CCN} formulae discussed in \cite{THIKONOV} for Toeplitz structure.
\end{remark}
\section{Numerical experiments}\label{sec7}
In order to check the reliability of the proposed structured \textit{CNs}, we perform  {several} numerical experiments in this section. All numerical tests are conducted on MATLAB R2023b on an Intel(R) Core(TM) $i7$-$10700$ $CPU$, $ 2.90GHz,$ $ 16$ $GB$ memory with machine precision $\mu=2.2\times 10^{-16}.$   

{ We construct the perturbations to the input data as follows: 
\begin{eqnarray}\label{per1}
   && \D A= 10^{-q} \cdot \D A_1\odot A, \quad \D B= 10^{-q} \cdot \D B_1\odot B,\quad \D C= 10^{-q} \cdot \D C_1\odot C,\\ \label{per2}
   &&\D D= 10^{-q} \cdot \D D_1\odot D,
		 \,\D f= 10^{-q}\cdot \D f_1\odot f, \, \D g= 10^{-q} \cdot \D g_1\odot g,
\end{eqnarray}
where $\D A_1\in \R^{m\times m},\D B_1, \D C_1\in \R^{n\times m}$ and 	$\D D_1\in \R^{n\times n}$ are the random matrices, preserving the structures of original matrices. Here, $\odot$ represents the entrywise multiplication of two matrices of the same dimensions. 
  Suppose that $[x^{\T},\, y^{\T}]^{\T}$ and $[\tilde{x}^{\T},\, \tilde{y}^{\T}]^{\T}$ are the unique solutions of the original \textit{GSPP}  and the perturbed \textit{GSPP}, respectively.  To estimate an upper bound for the forward error in the solution,  the normwise, mixed, and componentwise relative errors in $\mathbf{L}[x^{\T},y^{\T}]^{\T},$ respectively, are defined  by: }
	\begin{align*}
		&relk=\frac{\|\mathbf{L}[\tilde{x}^{\T},\, \tilde{y}^{\T}]^{\T}-\mathbf{L}[x^{\T}, \, y^{\T}]^{\T}\|_2}{\|\mathbf{L}[x^{\T},\, y^{\T}]^{\T}\|_2}, \quad  relm=\frac{\|\mathbf{L}[\tilde{x}^{\T},\, \tilde{y}^{\T}]^{\T}-\mathbf{L}[x^{\T}, \, y^{\T}]^{\T}\|_{\infty}}{\|\mathbf{L}[x^{\T},\, y^{\T}]^{\T}\|_{\infty}}, \\
		& \quad \quad relc=\left\|\frac{\mathbf{L}[\tilde{x}^{\T},\, \tilde{y}^{\T}]^{\T}-\mathbf{L}[x^{\T},\, y^{\T}]^{\T}}{\mathbf{L}[x^{\T},\, y^{\T}]^{\T}}\right\|_{\infty}.
	\end{align*} 
  {
  The following quantities $$\eta_1 \cdot \K(\mathbf{L}[x^{\T},\, y^{\T}]^{\T}), \quad \eta_2 \cdot \Ms(\mathbf{L}[x^{\T},\, y^{\T}]^{\T}), \quad \eta_2\cdot \C(\mathbf{L}[x^{\T},\, y^{\T}]^{\T}) \quad \text{and}$$
	$$\eta_1 \cdot \K(\mathbf{L}[x^{\T},\, y^{\T}]^{\T};\mathbb{S}), \quad \eta_2 \cdot \Ms(\mathbf{L}[x^{\T},\, y^{\T}]^{\T};\mathbb{S}), \quad \eta_2\cdot \C(\mathbf{L}[x^{\T},\, y^{\T}]^{\T};\mathbb{S}),$$ where $\mathbb{S}=\{\E, \mathcal{L}\},$ are the  estimated upper bounds of \textit{relk}, \textit{relm}, and \textit{relc} obtained by the \textit{CNs} in unstructured and structured cases,  respectively. Here, the quantities $\eta_1$ and $\eta_2$ are defined as \cite{MTLS2022}:
 \begin{eqnarray*}
	    \eta_1=\left\{\begin{array}{cc}
	      \frac{\left\|\bmatrix{\D H& \D \mathbf{d}}\right\|_F}{\left\|\bmatrix{H& \mathbf{d}}\right\|_F},& \text{when}~ \mathbb{S}=\E,\\
	          \frac{\left\|\bmatrix{\D \M& \D \mathbf{d}}\right\|_F}{\left\|\bmatrix{\M& \mathbf{d}}\right\|_F},& \text{when}~ \mathbb{S}=\L,
	    \end{array}\right.
	\end{eqnarray*}
 and $\eta_2=\min\{\eta: \left|\bmatrix{\D \M& \D \mathbf{d}}\right|\leq \eta\left|\bmatrix{\M& \mathbf{d}}\right|\}$.} 
 We choose the matrix $\mathbf{L}$ as    $I_{m+n},$ $\bmatrix{I_m & {\bf 0}},$ and $\bmatrix{{\bf 0} & I_n},$ so that the \textit{CNs} for $[x^{\T},\, y^{\T}]^{\T},$ $x$ and $y,$ respectively, are obtained.

 \vone
\begin{example}\label{ex1}
   Consider the \textit{GSPP} (\ref{Eq22}), where the data matrices $A,B,D,f$ and $g$ are given as follows: $$A=\bmatrix{\epsilon_1&\epsilon_1&-0.01&10&10&-30&30&10&30\\
   \epsilon_1 &\epsilon_1& \epsilon_1&-0.01& 10&10&-30&30&10\\
  -0.01 &\epsilon_1 &\epsilon_1& \epsilon_1&-0.01& 10&10&-30&30\\
10  & -0.01 &\epsilon_1 &\epsilon_1& \epsilon_1&-0.01& 10&10&-30\\
10&10  & -0.01 &\epsilon_1 &\epsilon_1& \epsilon_1&-0.01& 10&10\\
-30 & 10 & 10  & -0.01 &\epsilon_1 &\epsilon_1& \epsilon_1&-0.01& 10\\
30&-30 & 10 & 10  & -0.01 &\epsilon_1 &\epsilon_1& \epsilon_1&-0.01\\
10&30&-30 & 10 & 10  & -0.01 &\epsilon_1 &\epsilon_1& \epsilon_1\\
30 &10&30&-30 & 10 & 10  & -0.01 &\epsilon_1 &\epsilon_1&}\in \SY_9,$$

$$B=\bmatrix{8&1&2 & 3&4 & 5& 6&7 &8\\
-0.02&8&1&2 & 3&4 & 5& 6&7\\
-0.03&-0.002&8&1&2 & 3&4 & 5& 6\\
-0.04&-0.03&-0.002&8&1&2 & 3&4 & 5\\
-0.05&-0.04&-0.03&-0.002&8&1&2 & 3&4 }\in \mathcal{T}^{4\times 9},$$
$$ D=0.05*std(B^{\T}).*randn(n,n),$$ $f=[1,\ldots,1,m-1,1]^{\T}\in \R^9,$ and $g=randn(n,1)\in \R^4,$
where  {$randn(m,n)$ denotes the $m\times n$ random matrix generated by the MATLAB command $randn$ and} $std(B^{\T})$ denotes the standard deviation of $B^{\T}.$ Here, $m=9$ and $n=4.$  {For the perturbations to the input data 
 constructed as in \eqref{per1}-\eqref{per2} with $q=7,$ $\D B_1\in \mathcal{T}^{n\times m}$ is a randomly generated Toeplitz matrix, $\D A_1=\frac{1}{2}(\widehat{A}+{ \widehat{A}}^{\T}),$  and $ \widehat{A} \in \R^{m\times m},\D D_1\in \R^{n\times n}$  are random matrices.} 
   \begin{table}[ht!]
			\centering
		\caption{Comparison of unstructured and structured \textit{NCN}, \textit{MCN}, and \textit{CCN} with their corresponding relative errors when  { $\mathbf{L}=I_{13}$}  for Example \ref{ex1}.}
			\label{tab1}
   \resizebox{15cm}{!}{
				\begin{tabular}{cccccccccc}
					\toprule
				 $\epsilon_1$& \large{$relk$}  &$\K([x^{\T},\, y^{\T}]^{\T})$	& $\K([x^{\T},\, y^{\T}]^{\T};\,\E)$ & \large{$relm$}  &$\Ms([x^{\T},\, y^{\T}]^{\T})$& $\Ms([x^{\T},\, y^{\T}]^{\T};\,\E)$ & \large{$relc$} &$\C([x^{\T},\, y^{\T}]^{\T})$ &  $\C([x^{\T},\, y^{\T}]^{\T};\,\E)$ \\
  \midrule 
   10 &7.2266e-07 &2.6621e+03&	2.5899e+03&	6.7684e-07&	1.3407e+02&	7.4069e+01&	7.5571e-06&	1.9545e+03&	9.1963e+02\\[1ex]
   $10^0$ &1.1562e-06 &2.5310e+03&	2.4517e+03&	8.2825e-07&	1.2620e+02&	6.1160e+01&	7.5375e-06&	1.1485e+03&	5.5659e+02\\[1ex]
     $10^{-1}$ &1.7604e-06 &2.9961e+03&	2.9210e+03&	1.7656e-06&	1.4518e+02&	7.4370e+01&	1.1800e-05&	1.1229e+03&	5.0019e+02\\	[1ex]
    $10^{-2}$ &5.2202e-07 &1.8120e+03	&1.7515e+03&	5.1969e-07&	1.1842e+02&	5.5675e+01&	5.1819e-06&	1.4140e+03&	8.8193e+02\\[1ex]
     $10^{-2}$ & 6.1932e-07&	3.1643e+03&	3.0814e+03&	5.5990e-07&	1.6201e+02&	8.4378e+01&	3.5558e-06&	1.2309e+03	&7.3506e+02\\[1ex]
   $10^{-3}$ & 2.4807e-06&	1.6465e+03&	1.5744e+03&	2.0678e-06&	1.1802e+02&	7.9648e+01&	1.3033e-05&	1.2165e+03&	7.7459e+02\\[1ex]
  $10^{-4}$ & 1.2229e-06&	3.0661e+03&	2.9578e+03&	1.3129e-06&	1.5980e+02&	9.0334e+01&	6.3654e-06&	1.2789e+03&	9.3070e+02\\[1ex]
     \bottomrule
    \end{tabular}
    }
    \end{table}
  \begin{table}[ht!]
			\centering
		\caption{Comparison of unstructured and structured \textit{NCN}, \textit{MCN}, and \textit{CCN}  with their corresponding relative errors when $\mathbf{L}=\bmatrix{I_9 & {\bf 0}}$ for Example \ref{ex1}.}
			\label{tab2}
   \resizebox{15cm}{!}{
				\begin{tabular}{cccccccccc}
					\toprule
				 $\epsilon_1$& \large{$relk$}  &$\K(x)$	& $\K(x;\,\E)$ & \large{$relm$} &$\Ms(x)$& $\Ms(x;\,\E)$ & \large{$relc$} & $\C(x)$ &  $\C(x;\,\E)$ \\
  \midrule 
   10 &7.9185e-07&	2.9377e+03&	2.8591e+03&	5.8903e-07&	1.1477e+02&	6.7631e+01&	7.5571e-06&	4.9310e+02&	2.6889e+02\\[1ex]
   $10^0$ &9.0322e-07&	2.0236e+03&	1.9605e+03&	9.7497e-07&	1.3501e+02&	6.2699e+01&	2.7439e-06	&3.3927e+02	&1.5957e+02\\[1ex]
     $10^{-1}$ &1.9559e-06	&3.4895e+03&	3.4010e+03&	1.8283e-06&	1.4073e+02&	7.9523e+01&	1.1800e-05&	1.1229e+03&	5.0019e+02\\	[1ex]
    $10^{-2}$ &3.6024e-07&	1.3749e+03&	1.3294e+03&	4.2107e-07&	9.5080e+01&	4.5826e+01&	5.9700e-07&	1.4140e+03	&8.8193e+02\\[1ex]
     $10^{-2}$ &6.4001e-07&	2.9172e+03&	2.8412e+03	&7.1460e-07&	1.5484e+02&	8.2478e+01&	2.9604e-06&	1.1599e+03&	6.2412e+02\\[1ex]
   $10^{-3}$ & 1.4527e-06&	7.9932e+02&	7.6489e+02&	2.1781e-06&	9.3409e+01&	6.0897e+01&	6.8343e-06&	2.9309e+02	&1.9108e+02\\[1ex]
  $10^{-4}$ & 8.9097e-07&	3.1005e+03&	2.9895e+03&	1.0788e-06&	1.3093e+02&	6.4355e+01&	4.5008e-06&	1.2789e+03	&9.3070e+02\\[2ex]
   \bottomrule
    \end{tabular}
    }
    \end{table}
     \begin{table}[ht!]
			\centering
		\caption{Comparison of unstructured and structured \textit{NCN}, \textit{MCN}, and \textit{CCN} with their corresponding relative errors when  $\mathbf{L}=\bmatrix{ {\bf 0} &  I_4}$ for Example \ref{ex1}.}
			\label{tab3}
   \resizebox{15cm}{!}{
				\begin{tabular}{cccccccccc}
					\toprule
				 $\epsilon_1$& \large{$relk$}  &$\K(y)$	& $\K(y;\,\E)$ & \large{$relm$} &$\Ms(y)$& $\Ms(y;\,\E)$ & \large{$relc$} & $\C(y)$ &  $\C(y;\,\E)$ \\
  \midrule 
   10 &7.2081e-07&	2.6548e+03&	4.5585e+02&	6.7684e-07&	1.3407e+02&	7.4069e+01&	5.3298e-06&	1.9545e+03&	9.1963e+02\\[1ex]
   $10^0$ &1.1662e-06&	2.5513e+03&	4.1391e+02&	8.2825e-07&	1.2620e+02&	6.1160e+01&	7.5375e-06&	1.1485e+03	&5.5659e+02\\[1ex]
     $10^{-1}$ &1.7531e-06&	2.9773e+03&	6.4241e+02&	1.7656e-06&	1.4518e+02&	7.4370e+01&	1.9676e-06&	1.9721e+02&	1.0083e+02\\	[1ex]
    $10^{-2}$ &5.3000e-07&	1.8351e+03&	3.2251e+02&	5.1969e-07&	1.1842e+02&	5.5675e+01&	5.1819e-06&	8.0176e+02	&4.6153e+02\\[1ex]
     $10^{-2}$ &6.1861e-07&	3.1725e+03&	5.2490e+02&	5.5990e-07&	1.6201e+02&	8.4378e+01&	3.5558e-06&	1.2309e+03	&7.3506e+02\\[1ex]
   $10^{-3}$ & 2.5575e-06&	1.7068e+03&	2.3662e+02&	2.0678e-06&	1.1802e+02&	7.9648e+01&	1.3033e-05&	1.2165e+03&	7.7459e+02\\[1ex]
  $10^{-4}$ &1.2378e-06&	3.0646e+03&	6.8237e+02&	1.3129e-06&	1.5980e+02&	9.0334e+01&	6.3654e-06&	2.9429e+02	&1.8874e+02\\[1ex]
   \bottomrule

    \end{tabular}
    }
    \end{table}
    
  The numerical results for different choices of $\epsilon_1$ are reported in Tables \ref{tab1}--\ref{tab3} using the formulae presented in Theorems \ref{Th31} and \ref{Th32}.  { The sizes of $\eta_1$ and $\eta_2$ are about $10^{-8}$ and $10^{-7},$ respectively, for all cases.}  It can be observed that structured \textit{CNs}, in all cases, are much smaller (almost one order less) than their unstructured counterparts. Moreover, the estimated upper bounds proposed by the \textit{CNs} for $relk,$ $relm$, and  $relc$ are sharper in the structured case than in the unstructured ones. Notably,  structured \textit{MCN} and \textit{CCN} give sharper bounds than \textit{NCN} on the relative error as they are of the same order or one order larger than $relm$ and $relc,$ respectively. This indicates that it is more preferable to  adopt $\Ms(\mathbf{L}[x^{\T},\, y^{\T}]^{\T};\E)$ and $\C(\mathbf{L}[x^{\T},\, y^{\T}]^{\T};\E)$ to measure the true conditioning of the \textit{GSPP} \eqref{Eq22}.
\end{example}
\vspace{2mm}
{ \begin{example}\label{exam2}
In this example, we consider the \textit{GSPP} \eqref{Eq22} arising from the \textit{WTRLS} problem \cite{Benziimage2006}. Here $m=n$ and the Toeplitz matrix $B$ is given as follows:
\vspace{-2mm}
\begin{eqnarray*}
    B=[b_{ij}]\in \mathcal{T}^{n\times n}~~\quad\text{with}~~ b_{ij}=\frac{1}{\sqrt{2\pi\sigma}}e^{-\frac{(i-j)^2}{2\sigma^2}},
\end{eqnarray*}
 $A\in \R^{n\times n}$ is set to be a positive diagonal random matrix, and $D=-\nu I_n$ ($\nu>0$). The right hand side vector is taken as $\mathbf{d}=randn(2n,1)\in \R^{2n}.$

We select $\sigma=2$ and $\nu=0.001$ as in $[10]$. We set $q=8$ and construct perturbation matrices as in Example $6.1$. In all cases, we observed $\eta_1\approx \mathcal{O}(10^{-9})$ and $\eta_2\approx \mathcal{O}(10^{-8})$. The numerical results for structured and unstructured \textit{NCN}, \textit{MCN}, and \textit{CCN}, and the exact relative errors are reported in Tables \ref{exam2:tab1}-\ref{exam2:tab3} for different values of $n.$  
 \begin{table}[ht!]
			\centering
		\caption{   Comparison of unstructured and structured \textit{NCN}, \textit{MCN}, and \textit{CCN} with their corresponding relative errors  when $\mathbf{L}=I_{2n}$ for Example \ref{exam2}.}
			\label{exam2:tab1}
   \resizebox{15cm}{!}{
		  		\begin{tabular}{cccccccccc}
					\toprule
				 $n= m$&\large{$relk$}  &$\K([x^{\T},\, y^{\T}]^{\T})$	& $\K([x^{\T},\, y^{\T}]^{\T};\,\E)$ & \large{$relm$} &$\Ms([x^{\T},\, y^{\T}]^{\T})$& $\Ms([x^{\T},\, y^{\T}]^{\T};\,\E)$ & \large{$relc$} &$\C([x^{\T},\, y^{\T}]^{\T})$ &  $\C([x^{\T},\, y^{\T}]^{\T};\,\E)$ \\
  \midrule 
 $50$&4.1808e-07&	2.8177e+04&	2.4798e+04&	4.6643e-07&	1.4438e+03	&5.3588e+02&	3.3769e-05&	5.7501e+04	&2.0978e+04\\[1ex]
   $100$ &2.4188e-07&	4.8911e+03&	4.6982e+03&	2.5583e-07&	1.4305e+02&	4.1253e+01	&1.4497e-05&	1.1440e+04&	2.4661e+03\\[1ex]
     $150$ &5.3749e-07&	1.9378e+04&	1.7985e+04&	6.1184e-07&	5.5108e+02&	1.3986e+02	&2.4998e-04	&3.6099e+05	&8.1050e+04\\	[1ex]
    $200$ &7.5206e-07&	3.2373e+04&	9.4706e+03&	8.8297e-07&	1.0386e+03&	4.5302e+02&	9.7741e-05&	2.0373e+05&	4.4730e+04\\[1ex]
     \bottomrule
    \end{tabular}
    }
    \end{table}
     \begin{table}[ht!]
			\centering
		\caption{   Comparison of unstructured and structured \textit{NCN}, \textit{MCN}, and \textit{CCN} with their corresponding relative errors  when $\mathbf{L}=\bmatrix{I_{n}& \bf 0}$ for Example \ref{exam2}.}
			\label{exam2:tab2}
   \resizebox{15cm}{!}{
		  		\begin{tabular}{cccccccccc}
					\toprule
				 $n=m$&\large{$relk$}  &$\K([x^{\T},\, y^{\T}]^{\T})$	& $\K([x^{\T},\, y^{\T}]^{\T};\,\E)$ & \large{$relm$} &$\Ms([x^{\T},\, y^{\T}]^{\T})$& $\Ms([x^{\T},\, y^{\T}]^{\T};\,\E)$ & \large{$relc$} &$\C([x^{\T},\, y^{\T}]^{\T})$ &  $\C([x^{\T},\, y^{\T}]^{\T};\,\E)$ \\
  \midrule 
 $50$&3.8496e-07	&2.7536e+04&	2.4281e+04	&4.1735e-07&	1.2042e+03&	4.6611e+02&	3.2289e-06&	1.0158e+04	&3.1020e+03\\[1ex]
   $100$ &3.0293e-07&	7.0491e+03&	6.7372e+03&	3.7151e-07&	2.7486e+02&	7.2328e+01&	6.2591e-06&	6.5226e+03	&1.2953e+03\\[1ex]
     $150$ &7.2376e-07&	2.7944e+04&	2.5692e+04	&8.4422e-07	&7.5098e+02&	2.0082e+02&	6.3056e-05&	3.6099e+05	&8.1050E+04\\	[1ex]
    $200$ &8.0141e-07&	3.6034e+04&	3.2041e+04&	7.4664e-07&1.0283e+03&	4.1526e+02&	9.7741e-05&	2.0067e+05&	4.4730e+04\\[1ex]
     \bottomrule
    \end{tabular}
    }
    \end{table}
    \begin{table}[ht!]
			\centering
		\caption{   Comparison of unstructured and structured \textit{NCN}, \textit{MCN}, and \textit{CCN} with their corresponding relative errors  when $\mathbf{L}=\bmatrix{ \bf 0& I_{n}}$ for Example \ref{exam2}.}
			\label{exam2:tab3}
   \resizebox{15cm}{!}{
		  		\begin{tabular}{cccccccccc}
					\toprule
				 $n=m$&\large{$relk$}  &$\K([x^{\T},\, y^{\T}]^{\T})$	& $\K([x^{\T},\, y^{\T}]^{\T};\,\E)$ & \large{$relm$} &$\Ms([x^{\T},\, y^{\T}]^{\T})$& $\Ms([x^{\T},\, y^{\T}]^{\T};\,\E)$ & \large{$relc$} &$\C([x^{\T},\, y^{\T}]^{\T})$ &  $\C([x^{\T},\, y^{\T}]^{\T};\,\E)$ \\
  \midrule 
 $50$&4.2087e-07&	2.8235e+04	&7.2137e+03&	4.6643e-07	&1.4438e+03	&5.3588e+02&	3.3769e-05&	5.7501e+04	&2.0978e+04\\[1ex]
   $100$ &2.3999e-07&	4.8471e+03	&1.1173e+03&	2.5583e-07	&1.4305e+02&	4.1253e+01&	1.4497e-05	&1.1440e+04&	2.4661e+03\\[1ex]
     $150$ &5.2664e-07&	1.8878e+04	&5.6962e+03	&6.1184e-07&	5.5108e+02	&1.3986e+02	&2.4998e-04&	9.0978e+04	&3.1977e+04\\	[1ex]
    $200$ &7.4779e-07&	3.2089e+04&	9.2643e+03	&8.8297e-07	&1.0386e+03	&4.5302e+02&5.4363e-05&	2.0373e+05	&3.4820e+04\\[1ex]
     \bottomrule
    \end{tabular}
    }
    \end{table}
  We use Theorem \ref{Th32} and Remark \ref{re:Th32} to compute the structured \textit{CNs} and Theorem \ref{Th31} to compute unstructured \textit{CNs}. The results presented in Tables \ref{exam2:tab1}-\ref{exam2:tab3} reveal that the structured  \textit{NCN}, \textit{MCM}, and \textit{CCN} are much smaller than the unstructured ones for all values $n$. Specifically, for large matrices (with dimensions of $\mathcal{M}$ taken up to $400$), the structured \textit{CNs} are approximately an order of magnitude smaller than unstructured ones, showcasing the superiority of proposed structured \textit{CNs}.
\end{example}
\vone
\begin{example}\label{exam3}
    In this example, we consider the \textit{GSPP} arising from the discretization of the following Stokes equation by upwind scheme \cite{ZZStokes}:
    \begin{align}\label{stokes}
        \left\{\begin{array}{cc}
         -\mu \Delta {\bf u}+\nabla p=\tilde{f},    &\text{in} ~{\bm \Omega} , \\
            \nabla\cdot {\bf u}=\tilde{g} & \text{in} ~{\bm \Omega},\\
            {\bf u}={\bf 0}, &\text{on}~ \partial {\bm \Omega},\\
            \int_{{\bm \Omega}}p(x)dx=0,&
        \end{array}\right.
    \end{align}
    where ${\bm \Omega}=(0,1)\times (0,1)\in \R^2,$ $\partial {\bm \Omega}$ is the boundary of ${\bm \Omega},$ $\mu$ is the viscosity parameter, $\Delta$ is the Laplace operator, $\nabla$ represents the gradient, $\nabla\cdot$ is the divergence, ${\bf u} $ is the velocity vector, and $p$ is the scalar function representing the pressure. By discretizing \eqref{stokes}, we obtain the \textit{GSPP} \eqref{eq11} with
    $A=\bmatrix{I_r\otimes T+T\otimes I_r& {\bf 0}\\{\bf 0} & I_r\otimes T+T\otimes I_r}\in \R^{2r^2\times 2r^2},$ $B^{\T}=\bmatrix{I_r\otimes G\\ G\otimes I}\in \R^{2r^2\times r^2},$ $C=-B,$ $D={\bf 0},$ where 
    \vspace{-1mm}
    \begin{equation*}
        T=\frac{\mu}{h^2}~\mathrm{tridiag}(-1,2,-1)\in \R^{r\times r} ~ \text{and}~G=\frac{1}{h}~\mathrm{tridiag}(-1,1,0)\in \R^{r\times r}.
\end{equation*}
    Here, $\mathrm{tridiag}(a,b,c)$ denotes the tridiagonal matrix with diagonal entries $b,$ sub-diagonal entries $a$ and super-diagonal entries $c.$ Note that,  for this test problem $\mu=0.1,$ $m=2r^2$ and  $n=r^2,$ and we choose  ${\bf d}=[f^{\T},g^{\T}]^{\T}$ so that the exact solution is $z=[1,1,\ldots, 1]^{\T}\in \R^{m+n}.$ To avoid making $A$ too sparse, we add $X=0.5(X_1+X_1^{\T})$  to $A,$ where $X_1=sprandn(m,n,0.1).$ Here, $sprandn(m,n,0.1)$ denotes $ m\times n$ sparse random matrix with a density of 0.1.
    \begin{table}[ht!]
			\centering
		\caption{   Comparison of unstructured and structured \textit{NCN}, \textit{MCN}, and \textit{CCN} with their corresponding relative errors  when $\mathbf{L}=I_{m+n}$ for Example \ref{exam3}.}
			\label{exam3:tab1}
   \resizebox{15cm}{!}{
		  		\begin{tabular}{cccccccccc}
					\toprule
				 $r$&\large{$relk$}  &$\widetilde{\K}(z)$	& $\K([x^{\T},\, y^{\T}]^{\T};\,\L)$ & \large{$relm$} &$\widetilde{\Ms}(z)$& $\Ms([x^{\T},\, y^{\T}]^{\T};\,\L)$ & \large{$relc$} &$\widetilde{\C}(z)$ &  $\C([x^{\T},\, y^{\T}]^{\T};\,\L)$ \\
  \midrule 
$ 3$&	4.6396e-08&	1.0866e+02&	1.0325e+02&	9.2530e-08&	1.0315e+02&	8.3160e+01	&9.2530e-08&	1.0315e+02	&8.3160e+01\\[1ex]
   $ 4$&	1.0295e-07&	1.2567e+03	&1.1946e+03	&4.0283e-07&	1.1158e+03&	8.9754e+02	&4.0283e-07&	1.1158e+03&	8.9754e+02	\\	[1ex]
    $5$&	1.3490e-07&	1.1905e+03&	1.1256e+03&	5.3926e-07&	1.2062e+03&	9.4423e+02	&5.3926e-07&	1.2062e+03&	9.4423e+02\\[1ex]
    $6$&	1.1442e-07&	1.4744e+03&	1.3738e+03	&3.8692e-07	&1.1110e+03&	8.5833e+02&	3.8692e-07	&1.1110e+03&8.5833e+02\\[1ex]
    $7$&	1.4617e-07	&2.5853e+03	&2.5366e+03	&5.2901e-07	&1.1384e+03	&9.1026e+02&	5.2901e-07&	1.1384e+03	&8.1026e+02\\[1ex]
$8$	& 5.1493e-08&	2.6605e+03&	2.1679e+03&	2.0993e-07&	1.0634e+03	&8.8527e+02	&2.0993e-07&	1.0634e+03&	8.8527e+02\\[1ex]
$9$& 7.7302e-08&	1.2791e+03&	1.0043e+03&	2.5382e-07&	1.0339e+03&	8.2775e+02&	2.5382e-07&	1.0339e+03	&8.2775e+02	\\[1ex]
$10$ &1.2621e-07&	1.5807e+04&	1.5205e+04	&4.5006e-07&	1.0406e+04&	8.3004e+03&	4.5006e-07&	1.0406e+04	&8.3004e+03\\[1ex]
     \bottomrule
    \end{tabular}
    }
    \end{table}
    
    The perturbations in the input data constructed as in \eqref{per1}-\eqref{per2} with $q=8,$ $\D A_1=\frac{1}{2}(\widehat{A}+{ \widehat{A}}^{\T}),$  where $ \widehat{A} \in \R^{m\times m}$  is random matrix.  The numerical result for the structured and unstructured \textit{NCN}, \textit{MCN}, and \textit{CCN} with $\mathbf{L}=I_{m+n} $ are presented in Table \ref{exam3:tab1} for $r=3,4,\ldots, 10.$ Since the block matrix $A$ is symmetric, we compute the structured \textit{NCN}, \textit{MCN} and \textit{CCN} using Theorem \ref{th31} and Remark \ref{RE:TH45} with $D={\bf 0}$.
Unstructured \textit{CNs}  are computed using \eqref{eq331}, \eqref{eq332},
 and Remark \ref{rm41}. We observed $\eta_1\approx \mathcal{O}(10^{-9})$ and $\eta_2\approx \mathcal{O}(10^{-8})$ in all cases. Results reported in Table \ref{exam3:tab1} demonstrate that for all values of $r,$ structured \textit{MCN} and \textit{CCN} are almost one order smaller than the unstructured \textit{MCN} and \textit{CCN}. Moreover, the estimated upper bounds of the relative error of the solution produced by the structured \textit{CNs} are sharper than those obtained by the unstructured \textit{CNs} irrespective of the increasing size of $\M$ (taken up to $300$).

    \end{example}
\vspace{2mm}
{ \begin{example}\label{exam4}   Consider  the following second-order ODE \cite{ ZZBai2014}:
     \begin{equation}\label{eq:ode2}
        \left\{\begin{array}{lc}
          u_1^{''}(t)-u_2^{'}(t) -\frac{1}{t}u_2(t)  ={\bf 0},\\
            u_1^{'}(t)-\frac{1}{t}u_1(t)+u_2^{''}(t) -\frac{2}{t^2}u_2(t)  =\sigma(t), & 0<t<1,\\
             u_1(0)=0,~ u_1(1)=0,~ u_1^{'}(0)=0 ~\text{and} ~u_2(0)=0,&
        \end{array}\right.
    \end{equation}
   where $\sigma(t)=3t^3-4t^2+13t-2/t,$ and  the exact solutions are  $u_1(t)=t^2(1-t)^2$ and $u_2(t)=3t^3-4t^2+t.$

     The sinc discretization \cite{ZZBai2014} of the second-order ODE system \eqref{eq:ode2} yields the \textit{GSPP} \eqref{eq11} (with $m=n$). The block matrices $A, B, C$, and $D$ are computed using the formulae provided in \cite[Page-114, 115]{ZZBai2014}. 
      To exploit the Toeplitz structures on $A$ and $D$,  we define them as follows:
    \begin{align*}
    &A=D=T_1+T_2\in \R^{n\times n}.
\end{align*}
Other block matrices are given by \begin{align*}
     &B^{\T}=\frac{1}{2}(K_1T_1+T_1K_1)+K_2\in \R^{n\times n} ~\text{and}~C=-\frac{1}{2}(K_1T_1+T_1K_1)+K_3\in \R^{n\times n},
\end{align*}
 where $T_1, T_2\in \R^{n\times n}$ are defined by
 \begin{eqnarray*}
     T_1=\bmatrix{0&-1&\frac{1}{2}&\ldots&\frac{(-1)^{n-1}}{n-1}\\
     1&0&\ddots&\ddots&\vdots\\ -\frac{1}{2}&1&\vdots&-1&\frac{1}{2}\\ \vdots&\ddots&\ddots&0&-1\\ -\frac{(-1)^{n-1}}{n-1}&\ldots&-\frac{1}{2}&1&0
     },~  T_2=\bmatrix{\frac{\pi^2}{3}&-2&\frac{2}{2^2}&\ldots&\frac{2(-1)^{n-1}}{(n-1)^2}\\
     -2&\frac{\pi^2}{3}&\ddots&\ddots&\vdots\\ \frac{2}{2^2}&-2&\vdots&-2&\frac{2}{2^2}\\ \vdots&\ddots&\ddots&\frac{\pi^2}{3}&-2\\ \frac{2(-1)^{n-1}}{(n-1)^2}&\ldots&\frac{2}{2^2}&-2&\frac{\pi^2}{3}
     },
 \end{eqnarray*}
 and
\begin{align*}
& K_1:=\mathbf{D}_{\chi_1},~ K^{i}_C:=\mathbf{D}_{\chi^i_C}, ~K^{i}_G:=\mathbf{D}_{\chi^i_{G}},
 K_i=\frac{1}{2}(K^{i}_C+K^{i}_G), \\
 & \chi_1=[g_1(t_{-N}),g_1(t_{-N+1}),\ldots, g_1(t_{N})]^{\T}, ~ \chi_C^i:=[g^{i}_C(t_{-N}),g^{i}_C(t_{-N+1}),\ldots, g^{i}_C(t_{N})]^{\T},\\
& \chi_G^i:=[g^{i}_G(t_{-N}),g^{i}_G(t_{-N+1}),\ldots, g^{i}_G(t_{N})]^{\T},~~i=2,3,\\
  &g_{1}=h\frac{\mu_1}{\phi'}, ~ g^{2}_C=g^{3}_C=-h^2\frac{\mu_0}{(\phi')^2},~g^{2}_G=-h^2\left(\frac{1}{\phi'}\left(\frac{\mu_1}{\phi'}\right)'+\frac{\mu_0}{(\phi')^2}\right),\\
  & \text{and}~g^{3}_G=h^2\left(\frac{1}{\phi'}\left(\frac{\mu_1}{\phi'}\right)'-\frac{\mu_0}{(\phi')^2}\right).
 \end{align*}
 The functions $\mu_1(t), \mu_0(t),$ and $\phi(t)$ are given by $\mu_1(t)=1,$ $\mu_0(t)=-\frac{1}{t},$  and $\phi(t)=\text{ln}(t/(1-t)).$     Moreover, $n=2N+1,$ $h=\pi/\sqrt{2N},$ and $t_k=\phi^{-1}(kh).$  
 \begin{table}[ht!]
			\centering
		\caption{   Comparison of unstructured and structured \textit{NCN}, \textit{MCN}, and \textit{CCN} with their corresponding relative errors  when $\mathbf{L}=I_{2n}$ for Example \ref{exam4}.}
			\label{exam4:tab1}
   \resizebox{15cm}{!}{
		  		\begin{tabular}{cccccccccc}
					\toprule
				 $n$&\large{$relk$} &$\widetilde{\K}(z)$	& $\K([x^{\T},\, y^{\T}]^{\T};\,\L)$ & \large{$relm$} &$\widetilde{\Ms}(z)$& $\Ms([x^{\T},\, y^{\T}]^{\T};\,\L)$ & \large{$relc$} &$\widetilde{\C}(z)$ &  $\C([x^{\T},\, y^{\T}]^{\T};\,\L)$ \\
  \midrule 
 $41$&2.1668e-07&	8.4401e+02&	8.3043e+02&	1.3862e-06&	1.4301e+02&	5.2023e+01	&9.6646e-05	&1.1794e+04	&5.2908e+03	\\[1ex]
   $81$ &1.6126e-07&	1.5963e+03&	1.2802e+03&	1.2767e-07&	2.0050e+02&	3.0540e+01&	2.5905e-05&	4.8140e+04	&1.0074e+03\\[1ex]
     $121$ &4.4422e-07&	3.2151e+03&	3.2054e+03&	3.3468e-07	&4.0321e+02&	9.9676e+01	&5.6058e-06&	1.7998e+04&	8.6974e+03\\	[1ex]
    $161$ &4.6746e-07&	6.4824e+03&	4.4279e+03&	4.2745e-07&	8.7713e+02&	9.7876e+01	&1.8302e-05&	1.1962e+05	&9.1779e+04\\[1ex]
    $201$&1.1583e-06&	1.0882e+04&	1.0725e+04&	9.4135e-07	&1.0197e+03	&7.0819e+02&	1.2433e-05&	1.2847e+05&	8.9705e+04\\[1ex]
     \bottomrule
    \end{tabular}
    }
    \end{table}
   Further, we select $f={\bf 0}\in \R^n$ and $g=randn(n,1)\in \R^{n}$ using the MATLAB command \textit{randn}. 
   
   We take $q=7,$ and generate the perturbation matrices as in $(6.1)$-$(6.2)$ with $\D A$ and $\D D$ being Toeplitz. The numerical results for the structured and unstructured \textit{CNs} and the relative errors for  $N=20,40,60,80,100$ are reported in Table \ref{exam4:tab1}. We find that for all cases, $\eta_1$ and $\eta_2$  approximately are of order $10^{-8}$ and $10^{-7},$  respectively. We compute the structured \textit{NCN}, \textit{MCN}, and \textit{CCN} using Theorem \ref{th31}, where the block matrices $A$ and $D$ have Toeplitz structures. Unstructured \textit{CNs}  are computed using \eqref{eq331}, \eqref{eq332},
 and Remark \ref{rm41}. The solution $z=[x^{\T},y^{\T}]^{\T}\in \R^{2n}$ is computed using MATLAB function $\mathcal{M} \backslash \mathbf{d},$ where $\mathbf{d}=[f^{\T},g^{\T}]^{\T}\in \R^{2n}.$  Results in Table \ref{exam4:tab1} illustrate that, for all values of $N,$ structured \textit{CNs}  provide sharper upper bounds for relative error in the solution. Furthermore, the structured \textit{MCN} and \textit{CCN} are nearly an order of magnitude smaller than the unstructured ones, even as the size of the matrix $\M$ increases (taken up to $402$).

\end{example}}
    
}
\section{Conclusions}\label{sec8}
In this paper, by considering structure-preserving perturbations on the block matrices, we have investigated structured \textit{NCN}, \textit{MCN}, and \textit{CCN} for the linear function $\mathbf{L}[x^{\T},\, y^{\T}]^{\T}$ of the solution of \textit{GSPPs}. We present compact formulae of structured \textit{CNs} for  $\mathbf{L}[x^{\T}, \, y^{\T}]^{\T}$   in two cases. First, when  $B=C$ is Toeplitz and  $A$ is symmetric.  Second, when $B\neq C$ and the matrices $A$ and $D$ possess linear structures. Furthermore, we have obtained unstructured \textit{CNs}' formulae for $B=C$, which generalizes the previous results on \textit{CNs} of \textit{GSPP} when $\mathbf{L}$ is $I_{m+n}, \bmatrix{I_m & \bf 0}$ and $\bmatrix{{\bf 0} & I_n}.$  Additionally, the relations between structured and unstructured \textit{CNs} are obtained. It is found that the structured \textit{CNs} are always smaller than their unstructured counterparts.   An application of obtained structured \textit{CNs} formulae is provided to find the structured \textit{CNs} for \textit{WTRLS} problems, and they are also used to retrieve some prior found results for Tikhonov regularization problems. Numerical experiments are performed to validate the theoretical findings pertaining to proposed structured \textit{CNs}.  Moreover, empirical investigations indicate that the proposed structured 
\textit{MCN} and \textit{CCN} give much more accurate error estimations to the solution of \textit{GSPPs} compared to unstructured \textit{CNs}.
\section*{Acknowledgments}
Pinki Khatun acknowledges the Council of Scientific $\&$ Industrial Research (CSIR) in New Delhi, India, for their financial support in the form of a fellowship (File no. 09/1022(0098)/2020-EMR-I).
\bibliography{ref}
\bibliographystyle{abbrvnat}
\end{document}